\documentclass[11pt,twoside]{amsart}
\usepackage{mathrsfs}
\usepackage{graphicx}
\usepackage{mathrsfs}
\usepackage{amsmath}
\usepackage{amsthm}
\usepackage{amsfonts}
\usepackage{amssymb}
\usepackage{latexsym}
\usepackage[all]{xy}
\usepackage[colorlinks=true]{hyperref}
\hypersetup{urlcolor=blue, citecolor=red}

\date{}
\allowdisplaybreaks[4] \footskip=15pt
\renewcommand{\uppercasenonmath}[1]{}

\numberwithin{equation}{section} \theoremstyle{plain}
\newtheorem{theorem}{Theorem}[section]
\newtheorem{corollary}[theorem]{Corollary}
\newtheorem{lemma}[theorem]{Lemma}
\newtheorem{proposition}[theorem]{Proposition}
\theoremstyle{definition}
\newtheorem{definition}[theorem]{Definition}

\newtheorem{remark}[theorem]{Remark}
\newtheorem*{ack*}{ACKNOWLEDGEMENTS}

\newcommand{\pf}{\noindent\begin {proof}}
\newcommand{\epf}{\end{proof}}

\begin{document}
\begin{center}
{
{\bf\large Recollements induced by good silting objects}\\

\vspace{0.5cm}    Rongmin Zhu$^{1,*}$, Jiaqun Wei$^{2}$ \\

$^{1}$Department of Mathematics, Nanjing University, Nanjing 210093, China\\
$^{2}$School of Mathematics Sciences, Nanjing Normal University, Nanjing 210023, China
}
\end{center}
\title{}\maketitle\footnote[0]{*Corresponding author.

E-mail addresses: rongminzhu@hotmail.com, weijiaqun@njnu.edu.cn. }
\vspace{-3em}
 $$\bf  Abstract$$
\leftskip10truemm \rightskip10truemm \noindent  Let $U$ be a silting object in a derived category over a dg-algebra $A$, and let $B$ be the endomorphism
 dg-algebra of $U$. Under some appropriate hypotheses, we show that if $U$ is
good, then there exist a dg-algebra $C$, a homological epimorphism $B\rightarrow C$ and a recollement among the (unbounded) derived categories $\mathbf{D}(C,d)$ of $C$, $\mathbf{D}(B,d)$ of $B$ and $\mathbf{D}(A,d)$ of $A$. In particular, the kernel of the  left derived functor $-\otimes^{\mathbb{L}}_{B}U$ is triangle equivalent to the derived category $\mathbf{D}(C,d)$. Conversely, if $-\otimes^{\mathbb{L}}_{B}U$ admits a fully faithful left adjoint functor,
then $U$ is good. Moreover, we establish a criterion for the existence of a recollement of the derived category of a  dg-algebra relative to two derived categories of weak non-positive dg-algebras. Finally, some applications are given related to good cosilting objects,  good 2-term silting complexes, good tilting complexes and modules, which recovers a recent result by Chen and Xi. \\
\vbox to 0.3cm{}

\leftskip0truemm \rightskip0truemm
\bigskip

\section { \bf Introduction }
\leftskip0truemm \rightskip0truemm
\bigskip
Tilting theory generalizes the classical Morita theory of equivalences between
module categories. Note that a version of the Tilting Theorem
can be formulated at the level of derived category. In representation
theory, the results of Rickard \cite{ri89} and Keller \cite{ke94} on a derived Morita theory for rings show that compact tilting complexes guarantee the existence of derived equivalences.
We refer to \cite{ke07} for a recent survey.

All above mention equivalences are induced by compact objects (i.e. objects such that the induced Hom-covariant functor commutes with respect
to direct sums). Recently, infinitely generated tilting modules over arbitrary associated rings
have become of interest in and attracted increasing attentions towards understanding derived
categories and equivalences of general rings \cite{ba10,ba11,tr09}. In \cite{ba10}, it is shown that if $T$ is a good tilting module over a ring $A$, the  right derived functor $\mathbb{R}\mathrm{Hom}_{A}(T,-)$ induces an equivalence between the derived category
$\mathbf{D}(A)$ and a subcategory of the derived category $\mathbf{D}(B)$, where $B$ is the endomorphism algebra of $T$. Thus, in general, the right derived functor $\mathbb{R}\mathrm{Hom}_{A}(T,-)$  does not define a
derived equivalence between $A$ and $B$. Let $T$ be a tilting $A$-module with projective dimension one. In \cite{ch12}, Chen and Xi proved that if
the tilting module $T$ is good, then the triangulated category Ker$(T\otimes_{B}^{\mathbb{L}}-)$ is equivalent
to the derived category of a ring $C$, and therefore, there is a recollemment among the
derived categories of rings $A$, $B$ and $C$. Conversely, the existence of such a recollement
implies that the given tilting module $T$ is good.

Silting modules are generalizations of tilting ones and they were introduced in \cite{an16} as infinitely generated versions of support $\tau$-tilting modules.
In a route similar to the one followed by tilting modules and, more generally,
tilting complexes, a few authors extended the notion of silting
object to the unbounded derived category $\mathbf{D}(R)$ of a ring $R$.
 Wei introduced in \cite{we13} the notion of semi-tilting complexes, which is a small generalization of tilting complexes. Small semi-tilting complexes induce some
equivalences between dg derived categories. In \cite{br18}, Breaz and Modoi  defined big, small and good silting objects in $\mathbf{D}(A, d)$, where $A$ is a dg-algebra. Under some fairly general appropriate
hypotheses, they show that it induces derived equivalences between the
derived category over $A$ and a subcategory of the derived category of
dg-endomorphism algebra $B$ of $U$
$$\mathrm{\mathbb{R}Hom}_{A}(U,-) : \mathbf{D}(A, d)\leftrightarrows\mathcal{ K}^{\perp} : -\otimes^{\mathbb{L}}_{B}U,$$
where $\mathcal{K} = \mathrm{Ker}(-\otimes^{\mathbb{L}}_{B}U)$.

The main purpose of this paper is to give a characterization of the triangulated categories
$\mathrm{Ker}(-\otimes^{\mathbb{L}}_{B}U)$ for a good silting object $U$. Namely we show that if
the silting object $U$ is good, then the triangulated category $\mathrm{Ker}(-\otimes^{\mathbb{L}}_{B}U)$ is equivalent
to the derived category of a dg-algebra $C$, and therefore, there is a recollemment among the
derived categories of dg-algebra $A$, $B$ and $C$. Conversely, the existence of such a recollement
implies that the given silting object $U$ is good. More precisely:

\begin{theorem}\label{the4.8} Let $A$ be a dg-algebra, $U$ a silting object
 in $\mathbf{D}(A,d)$, and let $B=\mathrm{DgEnd}_{A}(U)$.

(1) If $U$ is good, then there exist a dg-algebra $C$ and a recollement of triangulated categories
$$\xymatrixcolsep{3pc}\xymatrix{\mathbf{D}(C,d)\ar[r]^{\lambda_{\ast}}
&\ar@<-3ex>[l]_{}
\ar@<3ex>[l]^{}\mathbf{D}(B,d)
\ar[r]^{\mathbf{G}}&\ar@<-3ex>[l]_{}\ar@<3ex>[l]^{}\mathbf{D}(A,d)}$$
such that $\lambda:B\rightarrow C$ is a homological epimorphism. Moreover, $C$ is weak non-positive if and only if $H^{i}(U\otimes^{\mathbb{L}}_{A}\mathbb{R}\mathrm{Hom}_{A}(U,A))=0$ for $i\geq 2$, or equivalently, $H^{i}(U\otimes^{\mathbb{L}}_{A}\mathbb{R}\mathrm{Hom}_{B^{op}}(U,B))=0$ for $i\geq 2$.

(2) If the triangle functor $\mathbf{G}:=-\otimes^{\mathbb{L}}_{B}U: \mathbf{D}(B,d)\rightarrow \mathbf{D}(A,d)$ admits a fully faithful left adjoint $j_{!} :\mathbf{D}(A,d) \rightarrow \mathbf{D}(B,d)$, then the given silting object is good.
\end{theorem}

A very general result about recollement among  derived categories of dg
categories is proved in \cite{ya12} and \cite{ni18}. Part of our results are contained in \cite[Corollaries 7.7 and 7.8]{ni18}. Therefore our work generalizes the known results from the perspective of silting theory.

A necessary and sufficient criterion has been given \cite{ko91}
for a (bounded) derived module category of an algebra to admit a recollement, with the subcategory and quotient category again being derived module categories of rings.  Later on, the criterion has been extended and modified so as to cover derived categories of dg-algebras and unbounded derived categories as well and to work for any differential graded ring \cite{jo91,ni09}. In this paper, we also interpret these results from the viewpoint of silting theory and establish a criterion for the existence of a recollement of the derived category of a  dg-algebra (not necessary weak non-positive) relative to two derived
categories of weak non-positive dg-algebras. More precisely:

\begin{theorem}\label{the1.2}Let $B$ be a dg-algebra with right dg-modules $P$ and $Q$. Then the following are
equivalent.

$(i)$ There is a recollement
$$\xymatrixcolsep{3pc}\xymatrix{\mathbf{D}(C,d)\ar[r]^{i_{\ast}}&\ar@<-3ex>[l]_{i^{\ast}}
\ar@<3ex>[l]^{i^{!}}\mathbf{D}(B,d)
\ar[r]^{j^{\ast}}&\ar@<-3ex>[l]_{j_{!}}\ar@<3ex>[l]^{j_{\ast}}\mathbf{D}(A,d)}$$
where $C$ and $A$ are weak non-positive dg-algebras, for which
\begin{center}
$i_{\ast}(C)\cong P$,\ \ $j_{!}(A)\cong Q$.
\end{center}

(ii) In the derived category $\mathbf{D}(B,d)$, the dg-module $P$ is partial silting, $Q$ is compact partial silting, $P^{\perp}\cap Q^{\perp}=0$, and $P\in Q^{\perp}$.
\end{theorem}

The paper is organized as follows. We will start in Section 2 with some basics about the dg-algebras and their derived categories. In Section 3, we construct
 a recollement induced by good silting objects. Moreover, we give a version of \cite[Theorem 3.3]{jo91} from the perspective of silting theory.
 In Section 4, we shall first establish a connection
between homological epimorphisms of dg-algebras and recollements of triangulated categories induced by silting objects, and then get a recollemment among derived categories of dg-algebras, where the left two terms of the recollement can be induced by a homological epimorphism of dg-algebras moving the hypotheses of $k$-flatness.
Finally, we prove our main results and investigate when the induced dg-algebra is weak non-positive.
In Section 5, we apply our main results to  good cosilting objects, good 2-term silting complexes and good tilting modules, and get the recollements induced by them. In this way, we recover and extend recent results by  Chen and Xi \cite{ch19}.

\bigskip
\section { \bf Preliminaries }

Now we introduce some notations and conventions used later in the paper.

\bigskip
\hspace{-0.4cm}\textbf{2.1} \emph{Differential graded algebras and Differential graded modules.}
\bigskip

A good reference for dg-algebras and their derived categories is the book \cite{ye17}.
Let $k$ be a commutative ring. Recall that a dg-algebra is a $\mathbb{Z}$-graded
$k$-algebra $A =\bigoplus_{i\in \mathbb{Z}} A^{i}$ endowed with a differential $d : A \rightarrow A$ such that
$d^{2} = 0$ which is homogeneous of degree 1, that is $d(A^{i})\subseteq A^{i+1}$ for all $i\in \mathbb{Z}$,
and satisfies the graded Leibniz rule:
\begin{center}
$d(ab) = d(a)b + (-1)^{i}ad(b)$, for all $a \in A^{i}$ and $b \in A$.
\end{center}
A (right) dg-module over $A$ is a $\mathbb{Z}$-graded module
$M =\bigoplus _{i\in Z}M^{i}$
endowed with a $k$-linear square-zero differential $d : M \rightarrow M$, which is
homogeneous of degree 1 and satisfies the graded Leibnitz rule:
\begin{center}
$d(ma) = d(m)a + (-1)^{i}md(a)$, for all $m \in M^{i}$ and $a \in A$.
\end{center}
Left dg-$A$-modules are defined similarly. A morphism of dg-$A$-modules is
an $A$-linear map $f : M \rightarrow N$ compatible with gradings and differentials. In
this way we obtain the category Mod($A,d$) of all dg-$A$-modules. It is an abelian category.

If $A$ is a dg-algebra, then the dual dg-algebra $A^{op}$ is defined as follows:
as graded $k$-modules $A^{op} = A$, the multiplication is given by $ab = (-1)^{ij}ba$
for all $a \in A^{i}$ and all $b\in A^{j}$ and the differential $d : A^{op} \rightarrow A^{op}$ is the
same as in the case of $A$. It is clear that a left dg-$A$-module $M$ is a right
dg-$A^{op}$-module with the ``opposite" multiplication $xa = (-1)^{ij}ax$, for all
$a\in A^{i}$ and all $x\in M^{j}$. It is not hard to see that an ordinary $k$-algebra  can be viewed as a dg-algebra concentrated in degree $0$.    A dg-algebra $A$ is
 called \emph{non-positive} if $A^{i}=0$ for all $i>0$. A dg-algebra $A$ is
 called \emph{ weak non-positive} if $H^{i}(A)=0$ for all $i>0$ and \emph{ weak positive} if $H^{i}(A)=0$ for all $i<0$.

For a dg-module $X \in \mathrm{Mod}(A, d)$ one defines (functorially) the following
$k$-modules $Z^{n}(X) = \mathrm{Ker}(X^{n}\rightarrow X^{n+1})$, $B^{n}(X) = \mathrm{Im}(X^{n-1} \rightarrow X^{n})$, and $H^{n}(X) = Z^{n}(X)/B^{n}(X)$, for all $n \in \mathbb{Z}$. We call $H^{n}(X)$ the $n$-th cohomology group of $X$. A morphism of dg-modules is called quasi-isomorphism if
it induces isomorphisms in cohomologies. A dg-module $X \in \mathrm{Mod}(A, d)$ is
called \emph{acyclic} if $H^{n}(X) = 0$ for all $n \in \mathbb{Z}$. A morphism of dg-$A$-modules
$f : X \rightarrow Y$ is called \emph{null homotopic} provided that there is a graded homomorphism $s : X \rightarrow Y$ of degree $-1$ such that $f = sd + ds$. The homotopy
category $\mathbf{K}(A, d)$ has the same objects as $\mathrm{Mod}(A, d)$ and the morphisms are
equivalence classes of morphisms of dg-modules, modulo the homotopy. It is
well known that the homotopy category is triangulated.
The derived category $\mathbf{D}(A,d)$ is obtained from $\mathbf{K}(A,d)$ by formally inverting
all quasi-isomorphisms.

Let now $A$ and $B$ be two dg-algebras and let $U$ be a dg-$B$-$A$-bimodule
(that is $U$ is a dg-$B^{op}$-$A$-module). For every $X \in \mathrm{Mod}(A, d)$, we can
consider the so called dg-Hom:
$$\mathrm{Hom}^{\bullet}_{A}(U,X) =\prod_{n\in \mathbb{Z}}\mathrm{Hom}^{n}_{A}(U,X)$$
with $\mathrm{Hom}^{n}_{A}(U,X) =\prod_{i\in\mathbb{Z}}\mathrm{Hom}_{A^{0}}(U^{i},X^{n+i})$, whose differentials are given
by $d(f)(x) = d_{Y} f(x) - (-1)^{n} fd_{X}(x)$ for all $f \in \mathrm{Hom}^{n}_{A}(X,Y)$.
 Then $\mathrm{Hom}^{\bullet}_{A}(U,X)$ becomes a dg-$B$-module, so we get a functor
$\mathrm{Hom}^{\bullet}_{A}(U,-): \mathrm{Mod}(A, d) \rightarrow \mathrm{Mod}(B, d).$
It induces a triangle functor
$\mathrm{Hom}^{\bullet}_{A}(U,-): \mathbf{K}(A, d) \rightarrow \mathbf{K}(B, d).$
A dg-$B$-module $U$ is called cofibrant if for every
acyclic dg-$B$-module $N$, we have $\mathrm{Hom}_{\mathbf{K}(B,d)}(U,N) = 0$. This is equivalent to
$\mathrm{Hom}_{\mathbf{D}(B,d)}(U,M) = \mathrm{Hom}_{\mathbf{K}(B,d)}(U,M)$
for all dg-$B$-modules $M$. Dually we define fibrant objects.
We define
$$\mathbb{R}\mathrm{Hom}_{A}(U,-) : \mathbf{D}(A, d) \rightarrow \mathbf{D}(B, d)$$
by $\mathbb{R}\mathrm{Hom}_{A}(U,X)\cong \mathrm{Hom}^{\bullet}_{A}(U',X) \cong \mathrm{Hom}^{\bullet}_{A}(U,X')$ where $U'$ is a cofibrant
replacement of $U$ and $X'$ is a fibrant replacement of $X$.

Let $Y \in \mathrm{Mod}(B, d)$. There exists a natural grading on the usual tensor
product $Y\otimes_{B}U$, which can be described as:
$$Y\otimes^{\bullet}_{B} U =\bigoplus_{n\in\mathbb{Z}}Y\otimes^{n}_{B}U,$$
where $Y\otimes^{n}_{B}U$ is the quotient of
$\bigoplus_{i\in\mathbb{Z}}Y^{i}\otimes_{B^{0}}U^{n-i}$ by the submodule generated by
$y\otimes bu- yb\otimes u$ where $y \in Y^{i}$, $u \in U^{j}$ and $b \in B^{n-i-j}$, for all
$i, j \in \mathbb{Z}$. Together with the differential
$d(y u) = d(y)u + (-1)^{i}yd(u)$, for all $y \in Y^{i}$; $u \in U,$
we get a dg-$A$-module inducing a functor $-\otimes_{B}^{\bullet}U : \mathrm{Mod}(B, d)\rightarrow \mathrm{Mod}(A, d)$
and further a triangle functor
$-\otimes_{B}^{\bullet}U : \mathbf{K}(B, d) \rightarrow \mathbf{K}(A, d)$. The left derived
tensor product is defined by
$Y\otimes_{B}^{\mathbb{L}}U= Y'\otimes_{B}^{\bullet}U\cong Y\otimes_{B}^{\bullet}U'$ where $Y'$
and $U'$ are cofibrant replacements for $Y$ and $U$ in $\mathbf{K}(B,d)$ and $\mathbf{K}(B^{op}, d)$ respectively. It induces a triangle functor
$$-\otimes_{B}^{\mathbb{L}}U : \mathbf{D}(B, d)\rightarrow \mathbf{D}(A, d)$$
which is the left adjoint of $\mathbb{R}\mathrm{Hom}_{A}(U,-)$.

\bigskip
\hspace{-0.4cm}\textbf{2.2} \emph{Dimension and triangulated subcategories.}
\bigskip

Let $\mathcal{C}$ be an additive category.
Throughout the paper, a full subcategory $\mathcal{B}$ of $\mathcal{C}$ is always assumed to be closed under isomorphisms. We denote by add($X$) the full subcategory of $\mathcal{C}$ consisting
of all direct summands of finite coproducts of copies of $X$. If $\mathcal{C}$ admits small coproducts, then we denote by Add($X$) the full
subcategory of $\mathcal{C}$ consisting of all direct summands of small coproducts of copies of $X$. Dually, if $\mathcal{C}$ admits small products, then Prod($X$) denotes the full subcategory of $\mathcal{C}$ consisting of all direct summands of small products of copies of $X$.

Let $\mathcal{D}$ be a triangulated category with the shift functor denoted by [1], and let $\mathcal{C}$ be a subcategory of $\mathcal{D}$. We define the full subcategories of $\mathcal{D}$:
\begin{align*}
\mathcal{C}^{\perp}:&=\{X\in \mathcal{D}\mid \mathrm{Hom}_{\mathcal{D}}(C,X[i])=0~\text{for all}~i\in\mathbb{Z}~\text{and}~C\in\mathcal{C}  \};\\
^{\perp}\mathcal{C}:&=\{X\in \mathcal{D}\mid \mathrm{Hom}_{\mathcal{D}}(X, C[i])=0~\text{for all}~i\in\mathbb{Z}~\text{and}~C\in\mathcal{C}  \};\\
\mathcal{C}^{\perp_{>0}}:&=\{X\in \mathcal{D}\mid \mathrm{Hom}_{\mathcal{D}}(U,X[i])=0~\text{for all}~i>0 ~\text{and}~C\in\mathcal{C}\};\\
^{\perp_{>0}}\mathcal{C}:&=\{X\in \mathcal{D}\mid \mathrm{Hom}_{\mathcal{D}}(X, C[i])=0~\text{for all}~i>0~\text{and}~C\in\mathcal{C}  \}.
\end{align*}
Consider an object $X \in \mathcal{D}$. Following \cite{we13}, we say that $X$ has the
$\mathcal{C}$-resolution dimension (respectively $\mathcal{C}$-coresolution dimension )$\leq n$, and we
write dim$_{\mathcal{C}} X \leq n$, (codim$_{\mathcal{C}} X \leq n$) provided that there is a sequence of
triangles
$$X_{i+1} \rightarrow C_{i} \rightarrow X_{i}\rightsquigarrow
~~\text{with}~~ 0 \leq i \leq n$$
$$(\text{respectively}~ X_{i}\rightarrow C_{i} \rightarrow X_{i+1}\rightsquigarrow
~~\text{with}~~ 0 \leq i \leq n)$$
in $D$, such that $C_{i} \in \mathcal{C}$, $X_{0 }= X$ and $X_{n+1} = 0$. We will write dim$_{\mathcal{C}} X < \infty$
(codim$_{\mathcal{C}} X < \infty$) if we can find a positive integer $n$ such that dim$_{\mathcal{C}} X \leq n$ (respectively, codim$_{\mathcal{C}} X \leq n$).

Given a class of objects $U$ in $\mathcal{D}$, we denote by Tria($U$) the smallest full triangulated
subcategory of $\mathcal{D}$ which contains $U$ and is closed under small coproducts. If $U$ consists of
only one single object $U$, then we simply write Tria($U$) for Tria($\{U\}$).

\bigskip
\hspace{-0.4cm}\textbf{2.3} \emph{Recollements and $\mathrm{TTF}$ triples.}
\bigskip

In this subsection we recall the notion of a recollement of triangulated categories and the TTF triples.  The standard reference is \cite{be82}.

Let $\mathcal{T}$, $\mathcal{T'}$ and $\mathcal{T''}$ be triangulated categories. A \emph{recollement} of $\mathcal{T}$ relative to $\mathcal{T}'$ and $\mathcal{T}''$ is defined by six triangulated functors as follows
$$\xymatrixcolsep{3pc}\xymatrix{\mathcal{T}'\ar[r]^{i_{\ast}=i_{!}}&\ar@<-3ex>[l]_{i^{\ast}}\ar@<3ex>[l]^{i^{!}}\mathcal{T}
\ar[r]^{j^{\ast}=j^{!}}&\ar@<-3ex>[l]_{j_{!}}\ar@<3ex>[l]^{j_{\ast}}\mathcal{T}''}$$
satisfying the following conditions:

(i) $(i^{\ast},i_{\ast}), (i_{!},i^{!}),(j_{!},j^{!})$ and $(j^{\ast}, j_{\ast})$ are adjoint pairs;

(ii) $i_{\ast}, j_{\ast}$ and $j_{!}$ are full embeddings;

(iii) $i^{!}j_{\ast}=0$, and hence, $j^{!}i_{!}=0$ and $i^{\ast}j_{!}=0$;

(iv) for any object $T \in \mathcal{T}$, there exist the following triangles in $\mathcal{T}$:
$$i_{!}i^{!}(T)\rightarrow T\rightarrow j_{\ast}j^{\ast}(T)\rightsquigarrow
 ~~~\ ~\text{and}~~~~\ \ j_{!}j^{!}(T)\rightarrow T\rightarrow i_{\ast}i^{\ast}(T)\rightsquigarrow.$$

Let $\mathcal{D}$ be a triangulated category. A \emph{torsion pair} in $\mathcal{D}$ is a pair $\mathcal{(X,Y)}$ of full subcategories $\mathcal{X}$ and $\mathcal{Y}$ of
$\mathcal{D}$ satisfying the following conditions:

(1) Hom$_{\mathcal{D}}\mathcal{(X,Y) }$= 0;

(2) $\mathcal{X}[1] \subseteq \mathcal{X}$ and $\mathcal{Y}[-1] \subseteq \mathcal{Y}$; and

(3) for each object $C \in \mathcal{D}$, there is a triangle
$$X^{C} \rightarrow C \rightarrow  Y^{C}\rightarrow X^{C}[1]$$
in $\mathcal{D}$ such that $X^{C} \in \mathcal{X}$ and $Y^{C} \in \mathcal{Y}$.

Let $\mathcal{D}$ be a triangulated category. A TTF \emph{triple} in $\mathcal{D}$ is a triple
 $\mathcal{(U,V,W)}$ of full subcategories $\mathcal{U}$, $\mathcal{V}$ and $\mathcal{W}$
 of $\mathcal{D}$ such that both $\mathcal{(U,V)}$ and $\mathcal{(V,W)}$ are torsion pairs.

\bigskip
\section { \bf Recollements induced by silting objects }

Let $A$ be a dg-algebra. Recall from \cite[Section 3]{br18} that an object $U \in \mathbf{D}(A, d)$ is called \emph{(pre)silting} provided that it satisfies (the first two of) the following conditions:

($S1$). dim$_{\mathrm{Add}(A)} U < \infty$;

($S2$). $U^{(I)} \in U^{\perp_{>0}}$ for every set $I$;

($S3$). codim$_{\mathrm{Add}(U)} A < \infty$.

An object $U \in \mathbf{D}(A, d)$ is called \emph{small (pre)silting} provided that it satisfies
(the first two of) the following conditions:

($s1$). dim$_{\mathrm{add}(A)} U < \infty$;

($s2$). $U \in U^{\perp_{>0}}$, for every set $I$;

($s3$). codim$_{\mathrm{add}(U)} A < \infty$.

Recall that an object $U \in \mathbf{D}(A, d)$ is called \emph{small} (or compact) if
Hom$_{ \mathbf{D}(A, d)}(X,-)$ commutes with coproducts. A small silting object is an object which is both silting and small. A silting object is called \emph{good} if the condition ($S3$) above can be replaced by $(s3)$.

For a silting object $U$ and an $n\in \mathbb{N}$, the conditions codim$_{\mathrm{Add}(U)} A \leq n$ and dim$_{\mathrm{Add}(A)} U \leq n$ are equivalent. Call
$n$-silting a silting object satisfying these equivalent conditions.

\begin{remark}
(1) Every small silting object is good.

(2) The notion of an $n$-silting object agrees to the $n$-semitilting complex in
\cite{we13} and to the $(n+1)$-silting complex in \cite{an16}.

(3) From \cite[Lemma 2.2]{br18}, we know that a good silting object $U$ is cofibrant both as an $A$ and a $B^{op}$ dg-module.

(4) Let $B=\mathrm{DgEnd}_{A}(U)$. By \cite[Section 3.3]{br18}, if $U$ is an $n$-silting dg-$A$-module, then $B$ is weak non-positive.
\end{remark}
\textbf{We are ready to fix some notations which will be used in this paper.}
Let $k$ be a commutative ring and let $A$ be a dg-algebra, $U\in \mathbf{D}(A,d)$  and $B=\mathrm{DgEnd}_{A}(U)$. Set
\begin{align*}
    & \mathcal{E}:=\{X\in \mathbf{D}(B,d)\mid
   \mathrm{Hom}_{\mathbf{D}(B,d)}(\mathbb{R}\mathrm{Hom}_{A}(U,A),X[i])=0~\text{for all}~i\in \mathbb{Z} \};\\
 &\mathbf{G}:=-\otimes^{\mathbb{L}}_{B}U: \mathbf{D}(B,d)\rightarrow \mathbf{D}(A,d); \  \ \ \   \    \mathbf{H}:=\mathrm{\mathbb{R}Hom}_{A}(U,-):\mathbf{D}(A,d)\rightarrow \mathbf{D}(B,d);\\
 &\mathcal{Y}:=\mathrm{Ker}(\mathbf{G}); \ \ \ \ \ \ \  \ \ \ \mathcal{Z}:=\mathrm{Im}(\mathbf{H}).
\end{align*}

First, we recall the following result, see \cite[Theorem 2.4]{br18}.
\begin{theorem}\label{the3.2} Consider a good silting object $U \in \mathbf{D}(A, d)$, denote $\mathcal{K} = \mathrm{Ker}(-\otimes^{\mathbb{L}}_{B}U)$. Then there is an equivalence of categories
$$\mathrm{\mathbb{R}Hom}_{A}(U,-) : \mathbf{D}(A, d)\leftrightarrows\mathcal{ K}^{\perp} : -\otimes^{\mathbb{L}}_{B}U$$
and the dg-algebra $B$ is weak non-positive.
Finally, if $U$ is a small silting object, then $\mathcal{ K}^{\perp} = \mathbf{D}(B, d)$.
\end{theorem}

Before giving our main result in this section, we need the following useful lemmas.  Denote  by $\langle \mathcal{C}\rangle$ the smallest triangulated category which contains $\mathcal{C}$.

\begin{lemma}\label{lem3.2} The dg-$B$-module $\mathbf{H}(A)$ is compact and cofibrant in $\mathbf{D}(B,d)$.
\end{lemma}
\begin{proof}Since $\mathrm{codim}_{\mathrm{add}(U)}A\leq n$, there is a sequence of triangles in $\mathbf{D}(A,d)$
 $$A_{i}\longrightarrow U_{i}\longrightarrow A_{i+1}\rightsquigarrow~~ \text{with}~~ 0\leq i\leq n, \eqno(3.1)$$
 such that $U_{i}\in \mathrm{add}(U)$, $A_{0}=A$ and $A_{i+1}=0$. Applying the triangle functor $\mathbf{H}$ to $(3.1)$, we get a sequence of triangles:
$$\mathbf{H}(A_{i})\rightarrow \mathbf{H}(U_{i})\rightarrow \mathbf{H}(A_{i+1})\rightsquigarrow ~~ \text{with}~~ 0\leq i\leq n. $$
 Since $\mathbf{H}(U_{i})\in \mathrm{add}(B)$, we have $\mathrm{codim}_{\mathrm{add}(B)}\mathbf{H}(A)\leq n$ and $\mathbf{H}(A)\in \langle\mathrm{add}(B)\rangle$. According \cite[Lemmas 21, 22]{mu06}, $\mathbf{D}(B,d)$ is compactly generated and the compact objects in $\mathbf{D}(B,d)$ form a thick and $\aleph_{0}$-localising subcategory. Then it is easy to see that $\mathbf{H}(A)$ is a compact and cofibrant object in $\mathbf{D}(B,d)$.
\end{proof}

\begin{lemma}\label{lem3.3a} Let $X$ and $Y$ be two dg-$A$-modules. If $X=A_{A}$ or $X\in \mathrm{add}(U_{A})$, then there is a quasi-isomorphism of dg-$k$-modules
$$\mathbb{R}\mathrm{Hom}_{A}(X,Y)\overset{\cong}{\rightarrow}
\mathbb{R}\mathrm{Hom}_{B}(\mathbb{R}\mathrm{Hom}_{A}(U,X),\mathbb{R}\mathrm{Hom}_{A}(U,Y)).$$
\end{lemma}
\begin{proof}In the first step we want to define a natural map
$$\alpha:\mathbb{R}\mathrm{Hom}_{A}(X,Y)\longrightarrow
\mathbb{R}\mathrm{Hom}_{B}(\mathbb{R}\mathrm{Hom}_{A}(U,X),\mathbb{R}\mathrm{Hom}_{A}(U,Y))$$
in $\mathbf{D}(B,d)$ as follows: Since $U$ is cofibrant, so is $X$. We obtain
$$\mathbb{R}\mathrm{Hom}_{A}(X,Y)=\mathrm{Hom}^{\bullet}_{A}(X,Y)=\bigoplus_{n\in \mathbb{Z}}\mathrm{Hom}^{n}_{A}(X,Y).$$
If $X=A$, by Lemma \ref{lem3.2}, $\mathbb{R}\mathrm{Hom}_{A}(U,A)\in \mathrm{add}(B)$ is a cofibrant dg-$B$-module.  If $X\in \mathrm{add}(U_{A})$, then $\mathbb{R}\mathrm{Hom}_{A}(U,X)\in \mathrm{add}(B)$ is a cofibrant dg-$B$-module. So in the both case, we have
$$\mathbb{R}\mathrm{Hom}_{B}(\mathbb{R}\mathrm{Hom}_{A}(U,X),\mathbb{R}\mathrm{Hom}_{A}(U,Y))=
\mathrm{Hom}^{\bullet}_{B}(\mathrm{Hom}^{\bullet}_{A}(U,X),\mathrm{Hom}^{\bullet}_{A}(U,Y)).$$
It is straightforward to check that the assignment $\alpha:h\mapsto[g \mapsto h \circ g]$ defines a map from $\mathrm{Hom}^{\bullet}_{A}(X,Y)$ to $\mathrm{Hom}^{\bullet}_{B}(\mathrm{Hom}^{\bullet}_{A}(U,X),\mathrm{Hom}^{\bullet}_{A}(U,Y)).$

For the second step, we show that $\alpha$ is a quasi-isomorphism. By Theorem \ref{the3.2}, the functor
$\mathbb{R}\mathrm{Hom}_{A}(U,-)$ is fully faithful. Then our claim follows from the isomorphisms
\begin{align*}
H^{i}\mathbb{R}\mathrm{Hom}_{A}(X,Y)\cong \mathrm{Hom}_{\mathbf{D}(A,d)}(X,Y[i])& \cong \mathrm{Hom}_{\mathbf{D}(B,d)}(\mathbb{R}\mathrm{Hom}_{A}(U,X),\mathbb{R}\mathrm{Hom}_{A}(U,Y)[i])\\
&\cong H^{i}\mathbb{R}\mathrm{Hom}_{B}(\mathbb{R}\mathrm{Hom}_{A}(U,X),\mathbb{R}\mathrm{Hom}_{A}(U,Y))
\end{align*}
for all $i\in \mathbb{Z}$.
\end{proof}
\begin{lemma}\label{lem3.3} The pair $\mathcal{(Y,Z)}$ is a torsion pair in $\mathbf{D}(B,d)$. Moreover, $\mathcal{Y}=\mathcal{E}$.
\end{lemma}
\begin{proof}In fact, for any object $Y\in\mathcal{Y}$ and $W\in \mathbf{D}(A,d)$, we have $\mathrm{Hom}_{\mathbf{D}(B,d)}(Y,\mathbb{R}\mathrm{Hom}_{A}(U,W))\cong\mathrm{Hom}_{\mathbf{D}(A,d)}(Y\otimes^{\mathbb{L}}_{B}U,W)
=\mathrm{Hom}_{\mathbf{D}(A,d)}(0,W)=0$ because the pair $(-\otimes^{\mathbb{L}}_{B}U,\mathbb{R}\mathrm{Hom}_{A}(U,-))$ is an adjoint pair of triangle functors. This shows $\mathrm{Hom}_{\mathbf{D}(B,d)}(\mathcal{Y,Z})=0$. Let $\eta:\mathrm{Id}_{\mathbf{D}(B,d)}\rightarrow HG$ be the unit adjunction, and let $\varepsilon:GH\rightarrow \mathrm{Id}_{\mathbf{D}(A,d)}$ be the counit adjunction. By Theorem \ref{the3.2}, we know that $\varepsilon$ is invertible. Then for any $M\in \mathbf{D}(B,d)$, the canonical morphism $\eta_{M}:M\rightarrow \mathbf{HG}(M)$ can be extended to a triangle in $\mathbf{D}(B,d)$:
$$M\stackrel{\eta_{M}}\longrightarrow \mathbf{HG}(M)\longrightarrow N\rightsquigarrow.$$
By applying the functor $\mathbf{G}$ to the above triangle, we obtain a triangle in $\mathbf{D}(A,d)$:
$$\mathbf{G}(M)\stackrel{\mathbf{G}(\eta_{M})}\longrightarrow \mathbf{GHG}(M)\longrightarrow \mathbf{G}(N)\rightsquigarrow.$$
Since $\varepsilon$ is invertible, we see that $\mathbf{G}(\eta_{M})$ is an isomorphism. This shows $\mathbf{G}(N)=0$, that is $N\in \mathcal{Y}$. Furthermore, since $\mathcal{Y}$ is a triangulated subcategory of $\mathbf{D}(B,d)$, we have $N[-1]\in\mathcal{Y}$. Then the following triangle
$$N[-1]\longrightarrow M\longrightarrow \mathbf{HG}(M)\rightsquigarrow$$
in $\mathbf{D}(B,d)$ with $\mathbf{HG}(M)\in \mathcal{Z}$ shows that $\mathcal{(Y,Z)}$ is a torsion pair.

 In the following, we shall prove that $\mathcal{Y=E}$.  Applying the exact functor $\mathbb{R}\mathrm{Hom}_{A}(-,U)$ to (3.1), we get a sequence of triangles in $\mathbf{D}(B^{op},d)$ of the form
 $$V_{i+1}\longrightarrow B_{i}\longrightarrow V_{i}\rightsquigarrow~~ \text{with}~~ 0\leq i\leq n, $$
 such that $B_{i}=\mathbb{R}\mathrm{Hom}_{A}(U_{i},U)\in \mathrm{add}(B)$, $V_{0}=U$ and $V_{n+1}=0$. For any $X\in \mathbf{D}(B,d)$, applying the exact functor $X\otimes^{\mathbb{L}}_{B}-$, we get  a sequence of triangles
  $$X\otimes^{\mathbb{L}}_{B}V_{i+1}\longrightarrow X\otimes^{\mathbb{L}}_{B}B_{i}\longrightarrow X\otimes^{\mathbb{L}}_{B}V_{i}\rightsquigarrow~~ \text{with}~~ 0\leq i\leq n. \eqno(3.2)$$
  Note that $X\otimes^{\mathbb{L}}_{B}\mathbb{R}\mathrm{Hom}_{B}(B,B)\cong \mathbb{R}\mathrm{Hom}_{B}(B,X)$. Because  $U_{i}\in \mathrm{add}(U)$, it is easy to see  $\mathbb{R}\mathrm{Hom}_{A}(U,U_{i})\in \mathrm{add}(B)$. Thus we deduce the existence of natural isomorphisms in $\mathbf{D}(A,d)$ by \cite[Lemma 1.2]{br18}:
    $$X\otimes^{\mathbb{L}}_{B}\mathbb{R}\mathrm{Hom}_{B}(\mathbb{R}\mathrm{Hom}_{A}(U,U_{i}),B)\cong \mathbb{R}\mathrm{Hom}_{B}( \mathbb{R}\mathrm{Hom}_{A}(U,U_{i}),X).$$
 By Lemma \ref{lem3.3a}, it is easy to check that $\mathbb{R}\mathrm{Hom}_{B}( \mathbb{R}\mathrm{Hom}_{A}(U,U_{i}),B)\cong \mathbb{R}\mathrm{Hom}_{A}(U_{i},U)=B_{i}$. Therefore, we have natural isomorphisms  $X\otimes^{\mathbb{L}}_{B}B_{i}\cong\mathbb{R}\mathrm{Hom}_{B}(\mathbf{H}(U_{i}),X)$.

 On the other hand, applying the exact functor $\mathbf{H}$ and $\mathbb{R}\mathrm{Hom}_{B}(-,X)$ to ($3.1$), we get  a sequence of triangles:
 $$\mathbb{R}\mathrm{Hom}_{B}(\mathbf{H}(A_{i+1}),X)\rightarrow \mathbb{R}\mathrm{Hom}_{B}(\mathbf{H}(U_{i}),X)\rightarrow \mathbb{R}\mathrm{Hom}_{B}(\mathbf{H}(A_{i}),X)\rightsquigarrow \eqno(3.3)$$
 with $0\leq i \leq n$. Compare the triangles in $(3.2)$ with $(3.3)$, we see that there exist natural isomorphisms $X\otimes^{\mathbb{L}}_{B}V_{i}\cong \mathbb{R}\mathrm{Hom}_{B}(\mathbf{H}(A_{i}),X)$ for all $0\leq i \leq n$. For $n=0$, we have $X\otimes^{\mathbb{L}}_{B}U\cong \mathbb{R}\mathrm{Hom}_{B}(\mathbf{H}(A),X)$. Note that $H^{i}(\mathbb{R}\mathrm{Hom}_{B}(\mathbf{H}(A),X))\cong \mathrm{Hom}_{\mathbf{D}(B,d)}(\mathbb{R}\mathrm{Hom}_{A}(U,A),X[i])$ for any $i \in \mathbb{Z}$. Thus $\mathcal{Y=E}$ follows.
\end{proof}

\begin{lemma} \label{lem3.4}Let $U$ be a good silting object
 in $\mathbf{D}(A,d)$. Denote $Q:=\mathbf{H}(A)$. Then the triple $(\mathrm{Tria}(Q), \mathrm{Ker}(G), \mathrm{Im}(H))$ is a $\mathrm{TTF}$ triple in $\mathbf{D}(B,d)$.
\end{lemma}
\begin{proof} By Lemma \ref{lem3.3}, we only need to show that $(\mathrm{Tria}(Q), \mathrm{Ker}(G))$ is a torsion pair. From Lemma \ref{lem3.2}, we get that $\mathbb{R}\mathrm{Hom}_{A}(U,X)\in \mathrm{add}(B)$ is a compact dg-$B$-module. Therefore, the subcategory $\mathcal{Y}=\{X\in \mathbf{D}(B,d)\mid
   \mathrm{Hom}_{\mathbf{D}(B,d)}(\mathbb{R}\mathrm{Hom}_{A}(U,A),X[i])=0~\text{for all}~i\in \mathbb{Z} \}$ is closed under direct sums and products. Thus, the existence of the TTF triple $\mathcal{(X,Y,Z)}$ in $\mathbf{D}(B,d)$ follows straightforward from \cite[Proposition 5.14]{be04}. Moreover, $\mathcal{X}=\mathrm{Ker}(\mathrm{Hom}_{\mathbf{D}(B,d)}(-,\mathcal{Y}))$ and $\mathcal{Y}=\mathrm{Ker}(\mathrm{Hom}_{\mathbf{D}(B,d)}(\mathcal{X},-)$.

   Next we shall prove $\mathcal{X}=\mathrm{Tria}(Q)$. Let $\mathcal{Y}':=\mathrm{Ker}(\mathrm{Hom}_{\mathbf{D}(B,d)}(\mathrm{Tria}(Q),-))$. By
    \cite[ Chapter III, Theorem 2.3, Chapter IV, Proposition 1.1]{be07}, we see that $(\mathrm{Tria}(Q),\mathcal{Y}')$ is a torsion pair in $\mathbf{D}(B,d)$. It is easy to see $\mathcal{Y}\subseteq \mathcal{Y'}$. In particular, we have $\mathrm{Hom}_{\mathbf{D}(B,d)}(Q,\mathcal{Y}')=0$, which yields $Q\in \mathcal{X}$. Therefore, $\mathrm{Tria}(Q)\subseteq \mathcal{X}$ since $\mathcal{X}$ is a full triangulated subcategory of $\mathbf{D}(B,d)$. On the other side, let $j_{!}:\mathcal{X}\rightarrow \mathbf{D}(B,d)$ and $k:\mathcal{Z}\rightarrow \mathbf{D}(B,d)$ be the canonical inclusions. Then the functor $j_{!}$ has a right adjoint functor $R:\mathbf{D}(B,d)\rightarrow\mathcal{X}$. Since $\mathcal{(X,Y,Z)}$ is a TTF triple in $\mathbf{D}(B,d)$, the functor $Rk:\mathcal{Z}\rightarrow \mathcal{X}$ is an equivalence. So the composition functor $Rk\mathbf{H}:\mathbf{D}(A,d)\rightarrow \mathcal{X}$ is an equivalence because $\mathbf{H}:\mathbf{D}(A,d)\rightarrow \mathcal{Z}$ is an equivalence. Since $\mathbf{H}(A)=Q\in \mathcal{X}$, we have $Rk\mathbf{H}(A)\rightarrow \mathcal{X}$ is an equivalence under which $\mathrm{Tria}(A)$ has the image $\mathrm{Tria}(Q)$. Since $Rk\mathbf{H}$ commutes with coproducts, we get that $\mathcal{X}=\mathrm{Tria}(Q)$ and $\mathcal{Y=Y}'$. Hence $(\mathrm{Tria}(Q), \mathrm{Ker}(\mathbf{G}), \mathrm{Im}(\mathbf{H}))$ is a TTF triple in $\mathbf{D}(B,d)$.
\end{proof}

In the following, we show that good silting objects also induce a recollement of derived category of dg-algebras.

\begin{theorem}\label{the3.5} Let $A$ be a dg-algebra, $U$ a good silting object
 in $\mathbf{D}(A,d)$. Let $j_{!}:\mathcal{X}\rightarrow \mathbf{D}(B,d)$ and $i_{\ast}:\mathrm{Ker}(\mathbf{G})\rightarrow \mathbf{D}(B,d)$ be the inclusions.  Then we have a recollement of triangulated categories
$$\xymatrixcolsep{3pc}\xymatrix{\mathbf{D}(E,d)\ar[r]^{}&\ar@<-3ex>[l]_{}
\ar@<3ex>[l]^{}\mathbf{D}(B,d)
\ar[r]^-{j^{!}}&\ar@<-3ex>[l]_-{j_{!}=inc}\ar@<3ex>[l]^-{j_{\ast}}}\mathrm{Tri}(Q)$$
together with a triangle equivalence $\mathbf{G}j_{\ast}: \mathrm{Tria}(Q)\rightarrow \mathbf{D}(A,d)$, such that $\mathbf{G}j_{\ast}j^{!}$ is naturally isomorphic to $\mathbf{G}$, where $E=\mathbb{R}\mathrm{Hom}_{B}(i^{\ast}(B),i^{\ast}(B))$.
\end{theorem}
\begin{proof}By Lemma \ref{lem3.4},  $(\mathrm{Tria}(Q), \mathrm{Ker}(\mathbf{G}), \mathrm{Im}(\mathbf{H}))$ is a TTF triple in $\mathbf{D}(B,d)$. By the correspondence between recollements and TTF triples (see, for example, \cite[Section 9.2]{ne01}), there exists a recollement:
$$\xymatrix{\mathcal{Y}\ar[r]^-{i_{\ast}}&\ar@<-3ex>[l]_-{i^{\ast}}
\ar@<3ex>[l]^-{i^{!}}\mathbf{D}(B,d)
\ar[r]^{j^{!}}&\ar@<-3ex>[l]_{j_{!}=inc}\ar@<3ex>[l]^{j_{\ast}}\mathrm{Tria}(Q).}\eqno(3.4)$$
We infer from ($3.4$) that $\mathrm{Im}(j_{\ast})=\mathrm{Ker}(\mathrm{Hom}_{\mathbf{D}(B,d)}(\mathrm{Ker}(\mathbf{G}),-))$ and that the functor $j_{\ast}:\mathrm{Tria}(Q)\rightarrow \mathrm{Im}(j_{\ast})$ is a triangle equivalence with the restriction of $j^{!}$ to $\mathrm{Im}(j_{\ast})$ as its quasi-inverse. On the other hand, it follows from Lemma \ref{lem3.3} that $\mathrm{Im}(\mathbf{H})=\mathrm{Ker}(\mathrm{Hom}_{\mathbf{D}(B,d)}(\mathrm{Ker}(\mathbf{G}),-))$ and the functor $H:\mathbf{D}(A,d)\rightarrow \mathrm{Im}(\mathbf{H})$ is a triangle equivalence with the restriction of $\mathbf{G}$ to $\mathrm{Im}(\mathbf{H})$ as its quasi-inverse. Consequently, $\mathrm{Im}(j_{\ast})=\mathrm{Im}(\mathbf{H})$ and the composition $\mathbf{G}j_{\ast}:\mathrm{Tria}(Q)\rightarrow \mathbf{D}(A,d)$ of $j_{\ast}$ with $G$ is a triangle equivalence. Therefore, we get the following diagram of functors:
$$\xymatrixcolsep{3pc}\xymatrix{
  \mathbf{D}(B,d)\ar[r]^{j^{!}} &  \mathrm{Tria}(Q) \ar@/_1.8pc/[l]_{j_{!}=inc}\ar@/^1.8pc/[l]^{j_{\ast}} \ar[r]^{j_{\ast}} &  \mathrm{Im}\mathbf{(H)}=\mathrm{Im}(j_{\ast})\ar[r]^{\mathbf{G}} \ar@/_1.8pc/[l]_{j^{!}}&\mathbf{D}(A,d)
  \ar@/^1.8pc/[l]^{\mathbf{H}}}.$$

For any $X\in \mathbf{D}(B,d)$, by the recollement ($3.4$), there exists a canonical triangle in $\mathbf{D}(B,d)$:
$$i_{\ast}i^{!}(X)\rightarrow X\rightarrow j_{\ast}j^{!}(X)\rightsquigarrow.$$
Since $\mathrm{Im}(i_{\ast}i^{!})=\mathrm{Im}(i_{\ast})=\mathrm{Ker}(\mathbf{G})$, $\mathbf{G}(X)\cong \mathbf{G}j_{\ast}j^{!}(X)$ in $\mathbf{D}(A,d)$.

On the other hand, since $\mathcal{Y}$ is closed under coproducts and $\mathbf{D}(B,d)$ is compact generated by $B$, by \cite[Chapter IV, Proposition 1.1]{be07}, we get that $i^{\ast}(B)$ is a compact generator of $\mathcal{Y}$. Furthermore, it is well known that the derived category $\mathbf{D}(B,d)$ of dg-algebra $B$ is an algebraic triangulated categories, that is, it is triangle equivalent to the stable category of some Frobenius category. Since the triangle functor $i_{\ast}:\mathcal{Y}\rightarrow \mathbf{D}(B,d)$ is fully faithful, we see that $\mathcal{Y}$ is an algebraic triangulated category by \cite[Lemma 7.5(3)]{he06}. Let $E:=\mathbb{R}\mathrm{Hom}_{B}(i^{\ast}(B),i^{\ast}(B))$. Thus, by Keller's theorem \cite[Theorem 8.7]{ke07}, there is a derived equivalence $\mathbf{D}(E,d)\simeq \mathcal{Y}$ which can be illustrated by the following diagram
$$\xymatrixcolsep{6pc}\xymatrix{\mathbf{D}(E,d)\ar@<1.2ex>[r]^-{i^{\ast}(B)\otimes_{E}^{\mathbb{L}}-}&
\ar@<1.2ex>[l]^-{\mathbb{R}\mathrm{Hom}_{B}(i^{\ast}(B),-)}
 \mathcal{Y}. }
$$
Combining the statements above, we obtain the desired recollement.
\end{proof}

Before ending this section, we shall give a necessary and sufficient criterion for the existence of a recollement of the derived category of a dg-algebra relative to two derived
categories of weak non-positive dg-algebras. We will need the following notations.

\begin{definition}Let $\mathcal{T}$ be a triangulated category with set indexed coproducts. An object $P$ of $\mathcal{T}$ will be called \emph{self-compact} if the restricted functor
$\mathrm{Hom}_{\mathcal{T}}(P,-)|_{\mathrm{Tria}(P)}$ respects set indexed coproducts.
\end{definition}
\begin{definition}Let $B$ be a dg-algebra. A right dg-$B$-module $P$ is
called\emph{ partial silting} if it is self-compact and semi-self orthogonal, i.e.
$$\mathrm{Hom}_{\mathbf{D}(B,d)}(P, P[n]) = 0~\text{for every~ } n>0.$$
\end{definition}

There are also versions of \cite[Theorem 1.6]{jo91} and \cite[Proposition 3.2]{jo91} from the perspective of silting theory.

\begin{lemma}\label{lem5.6}Let $B$ be a dg-algebra with a cofibrant right dg-module $P$ which is partial silting
in $\mathbf{D}(B,d)$, and let $E$ be the endomorphism dg-algebra of $P_{B}$. Then $E$ is weak non-positive, and there is an adjoint pair of functors
$$\xymatrixcolsep{6pc}\xymatrix{\mathbf{D}(E,d)\ar@<1.2ex>[r]^-{\iota_{\ast}={-\otimes_{E}^{\mathbb{L}}P}}&
\ar@<1.2ex>[l]^-{\iota^{!}=\mathbb{R}\mathrm{Hom}_{B}(P,-)}
 \mathbf{D}(B,d). }
$$
where $\iota_{\ast}$ is a full embedding with essential image
$\mathrm{Im}\iota_{\ast} = \mathrm{Tria}(P)$.
\end{lemma}
\begin{proof}Since $P$ is partial silting
in $\mathbf{D}(B,d)$, we have
$H^{i}(E)\cong \mathrm{Hom}_{\mathbf{D}(B,d)}(P,P[i])=0$ for all $i>0$, thus $E$ is weak non-positive. Moreover, as a consequence of Neeman-Thomason localization \cite[Theorem 2.1.2]{ne96}, we have  $\mathbf{D}(E,d) = \mathrm{Tria}(E)$. The rest statements follows by \cite[Theorem 1.6]{jo91}.
\end{proof}

\begin{lemma}\label{lem5.7}Let $B$ be a dg-algebra and let $Q$ be an cofibrant and  compact partial silting right dg-$B$-module  in $\mathbf{D}(B,d)$.
Consider the endomorphism dg-algebra $A$ of $Q$.
Then $A$ is weak non-positive and $Q$ becomes a dg-$A$-$B$-bimodule. Let $S = Q^{\ast} = \mathbb{R}\mathrm{Hom}_{B}(Q,B)$. Then there is a recollement
$$\xymatrixcolsep{7pc}\xymatrix{Q^{\perp}\ar[r]^-{i=inc}&\ar@<-4ex>[l]_-{l}
\ar@<4ex>[l]^-{r}\mathbf{D}(B,d)
\ar[r]^{\mathbb{R}\mathrm{Hom}_{B}(Q,-)\cong{-\otimes_{B}^{\mathbb{L}}S}}
&\ar@<-4ex>[l]_{-\otimes_{A}^{\mathbb{L}}Q}\ar@<4ex>[l]^{\mathbb{R}\mathrm{Hom}_{A}(S,-)}
\mathbf{D}(A,d).}\eqno(3.5)$$
\end{lemma}
\begin{proof}Since $Q$ is partial silting
in $\mathbf{D}(B,d)$, we get that $A$ is weak non-positive. The recollement follows from \cite[Lemmas 3.1 and 3.2]{ba13}.
\end{proof}
The preceding material allows me to prove the Theorem \ref{the1.2}. The idea of the proof of this theorem is
essentially taken from \cite[Theorem 3.3]{jo91}.

\bigskip
{PROOF OF THEOREM 1.2.}
$(i)\Rightarrow(ii)$. We claim that $P$ and $Q$ are partial silting. Since $C$ is a weak non-positive dg-algebra, we have $\mathrm{Hom}_{\mathbf{D}(B,d)}(P,P[i])\cong\mathrm{Hom}_{\mathbf{D}(C,d)}(C,C[i])\cong H^{i}\mathrm{Hom}^{\bullet}(C,C)=H^{i}(C)=0$ for all $i>0$. Thus $P$ is partial silting by \cite[Lemma 1.7]{jo91}. Similarly, $Q$ is a partial silting dg-module. The rest statements follows from the proof of \cite[Theorem 3.3]{jo91}.

$(ii)\Rightarrow(i)$.
One can clearly replace $P$ and $Q$ with cofibrant objects. Let $C$ and $A$ be the endomorphism
dg-algebras of $P_{B}$ and $Q_{B}$ so we have the full embedding of Lemma \ref{lem5.6} because $P_{B}$ is self-compact, and the recollement of Lemma \ref{lem5.7} because $Q_{B}$ is compact.
If we could prove
$$Q^{\perp}=\mathrm{Tria}(P),\eqno(3.6)$$
then we could replace $Q^{\perp}$ by $\mathrm{Tria}(P)$  which could again be replaced by $\mathbf{D}(C,d)$ using the
full embedding of Lemma \ref{lem5.6}, and this would give (3.5) which implies the desired recollement. In this case, we can get
$i_{\ast}(C) = i\iota_{\ast} (C)\cong C\otimes_{C}^{\mathbb{L}}P\cong P$ and $j_{!}(A)=A\otimes_{A}^{\mathbb{L}}Q\cong Q$.

To show (3.6), note that one part $``\supseteq"$  is clear since $P$ is in $Q^{\perp}$ by assumption while $Q$ is compact. To prove the other part $``\subseteq"$.
let $X$ be in $Q^{\perp}$. The adjunction in Lemma \ref{lem5.6} gives a counit
morphism $\iota_{\ast}\iota^{!}X\stackrel{\varepsilon}\rightarrow X$ (where $\iota_{\ast}$ and $\iota^{!}$ are now used in the sense of Lemma \ref{lem5.6}) which
can be extended to a distinguished triangle
$$\iota_{\ast}\iota^{!}X\stackrel{\varepsilon}\rightarrow X\rightarrow Y\rightsquigarrow.$$
By adjoint functor theory, $\iota^{!}(\varepsilon)$ is an isomorphism, so $\iota^{!}Y=0$, that is, $\mathbb{R}\mathrm{Hom}_{B}(P,Y) = 0$,
so $Y$ is in $P^{\perp}$.
Moreover, $\iota_{\ast}\iota^{!}X$ is in the essential image of $\iota_{\ast}$ which equals $\mathrm{Tria}(P)$ by Lemma \ref{lem5.6}. Since $\mathrm{Tria}(P)\subseteq Q^{\perp}$ it follows that $\iota_{\ast}\iota^{!}X$
 is in $Q^{\perp}$. But $X$ is in $Q^{\perp}$ by assumption,
and it follows that also $Y$ is in $Q^{\perp}$.
So $Y$ is in $P^{\perp}\cap Q^{\perp}$ which is 0 by assumption, thus $Y$ = 0. Therefore, the distinguished
triangle shows $X\cong\iota_{\ast}\iota^{!}X$ and it is in the essential image of $\iota_{\ast}$ which is equal to $\mathrm{Tria}(P)$.\hfill$\Box$

\bigskip
\section { \bf Homological epimorphisms and weak non-positive dg-algebras }

\textbf{4.1. The case of homological epimorphism of dg-algebras }
\bigskip

Chen and Xi \cite{ch12} considered the case of a good 1-tilting module $_{A}T$ with the endomorphism ring $B$. They showed that the left two terms of
a recollement as in the statement of \cite[Theorem 1.1]{ch12} are induced by a homological epimorphism of rings.
In the following, we recall the definition of homological epimorphisms of dg-algebras and its characterization at the level of derived categories.

\begin{definition}(see \cite[Theorem 3.9]{pa09}) Let $\lambda:C\rightarrow D$ be a morphism between two dg-algebras $C$ and $D$. Then $\lambda$ is called a homological epimorphism of dg-algebras if the canonical map $D\otimes^{\mathbb{L}}_{C}D\rightarrow D$ is an isomorphism, or equivalently, if the induced functor $\lambda_{\ast}:\mathbf{D}(D,d)\rightarrow \mathbf{D}(C,d)$ is fully faithful.
\end{definition}

In \cite[Theorem 5]{ni09}, it is proved that for a flat small dg category $\mathcal{B}$ there are bijections between
equivalence classes of recollements of $\mathbf{D}(\mathcal{B})$, TTF triples on $\mathbf{D}(\mathcal{B})$ and equivalence classes of homological epimorphisms of dg categories
$\lambda : \mathcal{B}\rightarrow\mathcal{C}$. Recall that a dg-$k$-algebra $B$ is flat if the functor $-\otimes_{k}B$ preserves acyclicity. Moreover, we have the following result in the case of a flat dg-$k$-algebra.

\begin{lemma}\label{lem4.2}(see \cite[Theorem 5]{ni09} or \cite[Lemma 2.14]{ba13}) Let $B$ be a dg-algebra flat as a $k$-module and $\mathcal{(X,Y,Z)}$ be a $\mathrm{TTF}$ triple in $\mathbf{D}(B,d)$. Then there is a
dg-algebra $C$ and a homological epimorphism $F:B\rightarrow C$ such that $\mathcal{Y}$ is the essential image of the restriction of scalars functor $F_{\ast}:\mathbf{D}(C,d)\rightarrow \mathbf{D}(B,d)$.
\end{lemma}

By the lemma above, we have the following result, which shows that  the left two terms of
a recollement as in the statement of Theorem \ref{the3.5} can be induced by a homological epimorphism of dg-algebras. From Lemma \ref{lem3.3a}, $\mathrm{DgEnd}_{B}(\mathbf{H}(A))\cong A$, then $\mathbf{H}(A)$ acquires the structure of dg-$A^{op}$-module.
\begin{theorem}\label{the4.3} Let $A$ be a dg-algebra, $U$ a good silting object
 in $\mathbf{D}(A,d)$. If $B$ is a $k$-flat dg-algebra, then there exist a dg-algebra $C$ and a recollement of triangulated categories
$$\xymatrixcolsep{7pc}\xymatrix{\mathbf{D}(C,d)\ar[r]^{\lambda_{\ast}}&\ar@<-4ex>[l]_{C\otimes_{B}^{\mathbb{L}}-}
\ar@<4ex>[l]^{\mathbb{R}\mathrm{Hom}_{B}(C, -)}\mathbf{D}(B,d)
\ar[r]^{\mathbf{G}\cong \mathbb{R}\mathrm{Hom}_{B}(\mathbf{H}(A),-)}&\ar@<-4ex>[l]_{-\otimes^{\mathbb{L}}_{A}\mathbf{H}(A)}\ar@<4ex>[l]^{\mathbf{H}}\mathbf{D}(A,d).}$$
such that $\lambda:B\rightarrow C$ is a homological epimorphism.
\end{theorem}
\begin{proof}By Lemma \ref{lem4.2}, there exist a dg-algebra $C$ and a homological epimorphism $\lambda:B\rightarrow C$, such that $\mathbf{D}(C,d)\overset{\sim}{\rightarrow} \mathcal{Y}$ is a triangulated equivalence. Then the existence of the recollement follows from  Theorem \ref{the3.5}.

In the following, we claim that $\mathbf{G}\cong \mathbb{R}\mathrm{Hom}_{B}(\mathbf{H}(A),-)$.
Let $P:=\mathbb{R}\mathrm{Hom}_{B}(\mathbf{H}(A),B)$. By Lemma \ref{lem3.2}, we get that $\mathbf{H}(A)$ is a compact and cofibrant object in $\mathbf{D}(B,d)$. Then it follows that the functors
$\mathbb{R}\mathrm{Hom}_{B}(\mathbf{H}(A),-)$ and $-\otimes^{\mathbb{L}}_{B}P$ are isomorphic by \cite[Section 2.1]{jo91}. So it is suffices to prove  $-\otimes^{\mathbb{L}}_{B}P\cong \mathbf{G}={-\otimes^{\mathbb{L}}_{B}U}$. In fact, by Lemma \ref{lem3.3a}, we have $P=\mathbb{R}\mathrm{Hom}_{B}(\mathbf{H}(A),B)\cong\mathbb{R}\mathrm{Hom}_{A}(A,U)\cong U$ and the claim follows.
\end{proof}

Next, we show that we can also get a homological epimorphism $F:B\rightarrow C$ for dg-algebras without assume that $B$ is a $k$-flat dg-algebra.

It is known that there is a projective model structure on the category Dga($k$) of dg-algebras over $k$ (see \cite[Theorem 4.1]{sc00} and \cite[Proposition 1.3.5(1)]{be14}). Denote by $\mathrm{HoDga(k)}$ the homotopy category of this model category. Note that every morphism $B \rightarrow C$ in the homotopy category $\mathrm{HoDga(k)}$ of dg-algebras over $k$ is by \cite[Remark A.6]{st17} and \cite[Proposition A.10]{st17} represented by a fraction
$$\xymatrix{
                & E \ar[dr]^{f} \ar[dl]_{\sigma}            \\
 B & &     C        }$$
where $\sigma:E\rightarrow B$ is a surjective quasi-isomorphism of dg-algebras and $E$ is cofibrant. Then the quasi-isomorphism $\sigma:E\rightarrow B$ induces a triangle equivalence $\sigma_{\ast}:\mathbf{D}(B,d)\rightarrow \mathbf{D}(E,d)$, whose quasi-inverse is $\sigma^{\ast}=- \otimes ^{\mathbb{L}}_{E}B:\mathbf{D}(E,d)\rightarrow \mathbf{D}(B,d)$ (see \cite[Lemma 6.1(a)]{ke94}). Hence we have a triangle functor
$$\xymatrixcolsep{4pc}\xymatrix{\mathbf{D}(C,d)\ar[r]^{\sigma^{\ast}f_{\ast}}& \mathbf{D}(B,d),}$$
which takes the role of the functor induced by the restriction of scalars. Then \cite[Theorem 3.9]{pa09} says that $\sigma^{\ast}f_{\ast}$ is full. Therefore the morphism $f\sigma^{-1}:B \rightarrow C$ in HoDga($k$) is a homological epimorphism. Moreover, in this case, $E$ is a $k$-flat dg-algebra.

Let $\mathcal{T}$ be a triangulated category. Recall that a \emph{Bousfield localization functor} \cite{kr10} is a triangulated endofunctor $L:\mathcal{T}\rightarrow \mathcal{T}$ together with a natural transformation $\eta:\mathrm{Id}_{T}\rightarrow L$ with $L\eta:L\rightarrow L^{2}$ being invertible and $L\eta=\eta L$. The objects in Ker$L$ are called $L$-acyclic. A Bousfield localization functor $L:\mathcal{T}\rightarrow \mathcal{T}$ is called \emph{smashing} if it preserves coproducts. A localizing class $\mathcal{X}\subseteq \mathcal{T}$ is called \emph{smashing} if it is the class of acyclic objects for a smashing localization functor.
We have the following useful lemma.

\begin{lemma}\label{lem4.4}(see \cite[Proposition 2.5]{st17}) Let $B$ be a dg-algebra over a commutative ring $k$ and $\mathcal{X}\subseteq \mathbf{D}(B,d)$ be a smashing localizing class. Then there is a homological epimorphism $g=f\sigma^{-1}:B \rightarrow C$ in $\mathrm{HoDga}(k)$, represented by homomorphisms of dg-algebras $\sigma:E\rightarrow B$ and $f:E\rightarrow C$, such that
$$\sigma^{\ast}f_{\ast} : \mathbf{D}(C,d)\rightarrow \mathbf{D}(B,d)$$
is fully faithful, its essential image coincides with $\mathcal{X}^{\perp}$ and
$\mathcal{X}=\{X\in \mathbf{D}(B,d) \mid X\otimes ^{\mathbb{L}}_{E}C=0\}$.
\end{lemma}

Next, we are ready to give a main result of this subsection, which shows that we can get a  recollemment among derived categories of dg-algebras, where the left two terms of the recollement can be induced by a homological epimorphism of dg-algebras moving the hypotheses of $k$-flatness. Furthermore, the existence of such a recollement implies that the given silting object $U$ is good.
 Recall that if $\mathrm{Add} \mathcal{C}\subseteq \mathcal{C}^{\perp_{i>0}}$, then by \cite[Proposition 2.5]{we13} we have
$$\langle \mathcal{C}\rangle = \{ X \in \mathcal{D}\mid  \mathrm{dim}_{\mathcal{C}}X < \infty\}[\mathbb{Z}] = \{ X \in \mathcal{D}\mid  \mathrm{codim}_{\mathcal{C}}X < \infty\}[\mathbb{Z}]$$
Hence this is the smallest thick subcategory (that is triangulated and closed under direct summands) containing $\mathcal{C}$.

\begin{theorem} \label{the4.5}
Let $A$ be a dg-algebra, $U$ a silting object
 in $\mathbf{D}(A,d)$, and let $B=\mathrm{DgEnd}_{A}(U)$.

(1) If $U$ is good, then there is a homological epimorphism $g=f\sigma^{-1}:B \rightarrow C$ in $\mathrm{HoDga}(k)$, represented by homomorphisms of dg-algebras $\sigma:E\rightarrow B$ and $f:E\rightarrow C$, such that the following is a recollement of triangulated categories:
$$\xymatrixcolsep{4pc}\xymatrix{\mathbf{D}(C,d)\ar[r]^{\sigma^{\ast}f_{\ast}}&\ar@<-3ex>[l]_{C\otimes_{E}^{\mathbb{L}}-}
\ar@<3ex>[l]^{\mathbb{R}\mathrm{Hom}_{E}(C, -)}\mathbf{D}(B,d)
\ar[r]^{\mathbf{G}}&\ar@<-3ex>[l]_{-\otimes^{\mathbb{L}}_{A}\mathbf{H}(A)}\ar@<3ex>[l]^{\mathbf{H}}\mathbf{D}(A,d).}$$

(2) If the triangle functor $\mathbf{G}:=-\otimes^{\mathbb{L}}_{B}U: \mathbf{D}(B,d)\rightarrow \mathbf{D}(A,d)$ admits a fully faithful left adjoint $j_{!} :\mathbf{D}(A,d) \rightarrow \mathbf{D}(B,d)$, then the given silting object is good.
\end{theorem}
\begin{proof}

(1) Define $\mathcal{X}:=\mathrm{Tria}(\mathbf{H}(A))$. By Lemma \ref{lem3.2}, $\mathbf{H}(A)$ is compact in $\mathbf{D}(B,d)$. Then $\mathbf{H}(A)$ is compact in $\mathcal{X}$, and $\mathcal{X}^{\perp}$ is closed under small coproducts by \cite[4.4.3]{ni07}. From Lemma \ref{lem3.4},  $(\mathcal{X},\mathcal{Y})$ is a torsion pair in $\mathbf{D}(B,d)$, then by \cite[4.4.16]{ni07}, there is a smashing localization functor $L$, such that $\mathcal{X}$ is the kernel of $L$. Applying Lemma \ref{lem4.4}, we obtain a homological epimorphism $g=f\sigma^{-1}:B \rightarrow C$ in $\mathrm{HoDga}(k)$, represented by homomorphisms of dg-algebras $\sigma:E\rightarrow B$ and $f:E\rightarrow C$, such that $\sigma^{\ast}f_{\ast} : \mathbf{D}(C,d)\rightarrow \mathcal{Y}$ is an triangle equivalence. Furthermore, it is not difficult to convince oneself that then $C\otimes_{E}^{\mathbb{L}}-$ and $\mathbb{R}\mathrm{Hom}_{E}(C,-)$ are, respectively, left and right adjoint of the functor $\sigma^{\ast}f_{\ast}$. Thus, the recollement follows from $(3.4)$ in Theorem \ref{the3.5} and Theorem \ref{the4.3}.

(2) Let $U$ be a silting object. Recall that $\mathbf{G}$ and $\mathbf{H}$ stand for the triangle functors $-\otimes^{\mathbb{L}}_{B}U: \mathbf{D}(B,d)\rightarrow \mathbf{D}(A,d)$ and $\mathbb{R}\mathrm{Hom}_{A}(U,-):\mathbf{D}(A,d)\rightarrow \mathbf{D}(B,d)$, respectively. Suppose that $\mathbf{G}$ admits a fully faithful left adjoint
$j_{!} :\mathbf{D}(A,d)\rightarrow \mathbf{D}(B,d)$. We want to show that $U$ is a good silting object.
To prove that $U$ is good, it suffices to show $\mathrm{codim}_{\mathrm{add}(U)}A< \infty$.

First, we observe some consequences of the assumption that $j_{!}$ is fully faithful. Set
$Q :=j_{!}(A)$. Since the left derived functor $\mathbf{G}$ commutes with coproducts, we can easily show that the functor $j_{!}$ preserves compact objects. In particular, the complex $Q$ is compact in $\mathbf{D}(B,d)$. Since the functor $j_{!}$ is fully faithful, we conclude from \cite[Chapter IV, Section 3,
Theorem 1, p.90]{ma97} that the unit adjunction morphism $\eta : \mathrm{Id}_{\mathbf{D}(A,d)}
\rightarrow \mathbf{G}j^{!}$ is invertible. Thus, $A \cong \mathbf{G}(Q)$ in $\mathbf{D}(A,d)$. It is known that $Q$ is compact if and only if $Q$ is a direct summand of an object of $\mathbf{D}(B,d)$ which is represented by a dg-module $P$ which has a finite filtration $F_{\bullet}$ by dg-submodules such that $F_{i}P/F_{i-1}P$ are finite direct sums of shifts of $B$ (see \cite[Tag 09QZ]{st16}). Consider the distinguished triangle
$$\bigoplus F_{i}P\rightarrow \bigoplus F_{i}P\rightarrow P\rightsquigarrow\eqno(4.1)$$
by \cite[Lemma 13.1]{st16}. For $j\geq 0$ there is a triangle
$$F_{j-1}P\rightarrow F_{j}P\rightarrow F_{j}P/F_{j-1}P\rightsquigarrow, $$
where $F_{-1}P=0$ and $F_{j}P/F_{j-1}P\cong \bigoplus^{r}_{i=1}B[k]$ for some $r\in\mathbb{N}$ and $k\in \mathbb{Z}$. Since $\langle \mathrm{add}(U)\rangle$ is closed under extensions, shift, direct summands and finite direct sums, we get that $\mathbf{G}(F_{i}P)=F_{i}P\otimes^{\mathbb{L}}_{B}U\in \langle \mathrm{add}(U)\rangle$ for all $i\geq 0$. Therefore, applying the functor $\mathbf{G}$ to $(4.1)$, we get a triangle
$$\bigoplus \mathbf{G}(F_{i}P)\rightarrow \bigoplus \mathbf{G}(F_{i}P)\rightarrow \mathbf{G}(P)\rightsquigarrow,$$
which implies that $\mathbf{G}(P)\in\langle \mathrm{add}(U)\rangle$.
Since $A$ is isomorphic to a direct summand of $\mathbf{G}(P)$, we get
$A\in \langle \mathrm{add}(U)\rangle$ by the `small' version of \cite[Corollary 4.2]{we13}. On the other hand, we can w.l.o.g. assume $\mathrm{dim}_{\mathrm{Add}(A)}U\leq n$, then $H^{i}(U)=0$ for $i \notin \{-n+1,\ldots,-1,0\}$ by \cite[Remark 2.1]{br18}. Consequently, we have $\mathrm{Hom}_{\mathbf{D}(A,d)}(A ,U[i])\cong \mathrm{H}^{i}\mathbb{R}\mathrm{Hom}_{A}(A,U)\cong \mathrm{H}^{i}(U)=0$ for all $i>0$, that is, $A\in{^{\perp_{>0}}\mathrm{add}(U)}\cap \langle \mathrm{add}(U)\rangle$. Therefore,  $\mathrm{codim}_{\mathrm{add}(U)}A< \infty$ by the `small' version of \cite[Corollary 2.6(1)]{we13}.
\end{proof}

\bigskip
\textbf{4.2. When the induced dg-algebra is weak non-positive }
\bigskip

Let $R$ be a dg-algebra, $\mathcal{T}$  a full triangulated subcategory of $\mathbf{D}(R,d)$. From \cite[Section 3]{ch19}, $\mathcal{T}$ is said to be \emph{bireflective}
if the inclusion $\mathcal{T} \rightarrow \mathbf{D}(R,d)$ admits both a left adjoint and a right adjoint, and \emph{homological}
if there is a homological  epimorphism $\lambda : R \rightarrow C$ of dg-algebras such that $\lambda_{\ast}$ induces a triangle equivalence from $\mathbf{D}(C,d)$ to $\mathcal{T}$.
In this subsection, we consider when the induced dg-algebra $C$ is weak non-positive.

First, we have the following easy observation.

\begin{lemma} \label{lem4.7} Let $A$ be a dg-algebra, and let $i^{\ast}$ be the left adjoint of the inclusion $i_{\ast}:\mathcal{Y}\rightarrow \mathbf{D}(B,d)$. Then the induced  dg-algebra $C$ in Theorem \ref{the4.5} is weak non-positive if and only if $H^{i}(i^{\ast}(B))=0$ for every $i\geq1$.

\end{lemma}
\begin{proof} Let $Y:=i^{\ast}(B)$. We only need to prove that $i^{\ast}(B)\cong C$. Since $\mathbf{D}(C,d)\rightarrow \mathcal{Y}$ is a triangle equivalence, we can w.l.o.g., view $Y$ as a dg-$C$-module.
Let $B\stackrel{\varphi}\rightarrow Y$ be the unit adjunction morphism associated to the adjoint pair $(i^{\ast},i_{\ast})$. Since $C$ belongs to $\mathcal{Y}$ we have that
$\mathrm{Hom}_{\mathcal{Y}}(Y,C)\cong \mathrm{Hom}_{\mathbf{D}(B,d)}(B,C)$. So there is a unique morphism $f:Y\rightarrow C$ such that $\lambda=f\circ\varphi$. On the other hand, since $\lambda:B\rightarrow C$ is a  homological epimorphism, the natural morphism $\mathbb{R}\mathrm{Hom}_{C}(C,Y)\rightarrow \mathbb{R}\mathrm{Hom}_{B}(C,Y)$ is an isomorphism. Therefore, we get that $$\mathrm{Hom}_{\mathbf{D}(B,d)}(C,Y)\cong\mathrm{Hom}_{\mathbf{D}(C,d)}(C,Y)\cong H^{0}Y\cong\mathrm{Hom}_{\mathbf{D}(B,d)}(B,Y).$$
 Hence, we conclude that there is $g\in \mathrm{Hom}_{\mathbf{D}(B,d)}(C,Y)$, such that $g\circ\lambda=\varphi$. Consequently, $g\circ f\circ\varphi=\varphi$ and $\lambda=f\circ g\circ\lambda$. Since $\mathcal{Y}$ is a triangulated subcategory and the inclusion functor of $\mathcal{Y}$ into $\mathbf{D}(B,d)$ admits both a left and a right adjoint, by \cite[Proposition 1.4]{lv18}, we get that $\mathcal{Y}$ is both covering and enveloping. We conclude that $f\circ g=\mathrm{id}_{C}$ and $g\circ f=\mathrm{id}_{Y}$. Thus $C\cong Y= i^{\ast}(B)$.
\end{proof}

The following is the another main result in this section.

\begin{theorem}\label{the4.8} Let $A$ be a dg-algebra, $U$ a good $n$-silting object
 in $\mathbf{D}(A,d)$, and let $B=\mathrm{DgEnd}_{A}(U)$. Then there exist a dg-algebra $C$ and a recollement of triangulated categories
$$\xymatrixcolsep{3pc}\xymatrix{\mathbf{D}(C,d)\ar[r]^{\lambda_{\ast}}
&\ar@<-3ex>[l]_{}
\ar@<3ex>[l]^{}\mathbf{D}(B,d)
\ar[r]^{\mathbf{G}}&\ar@<-3ex>[l]_{-\otimes^{\mathbb{L}}_{A}\mathbf{H}(A)}\ar@<3ex>[l]^{\mathbf{H}}\mathbf{D}(A,d)}$$
such that $\lambda:B\rightarrow C$ is a homological epimorphism. Moreover, $C$ is weak non-positive if and only if $H^{i}(U\otimes^{\mathbb{L}}_{A}\mathbb{R}\mathrm{Hom}_{A}(U,A))=0$ for $i\geq 2$, or equivalently, $H^{i}(U\otimes^{\mathbb{L}}_{A}\mathbb{R}\mathrm{Hom}_{B^{op}}(U,B))=0$ for $i\geq 2$.
\end{theorem}
\begin{proof} From Theorem \ref{the3.5} and its proof, there exists a recollement of derived categories
$$\xymatrixcolsep{4pc}\xymatrix{\mathcal{Y}\ar[r]^-{i_{\ast}=inc}
&\ar@<-3ex>[l]_-{i^{\ast}}
\ar@<3ex>[l]^-{i^{!}}\mathbf{D}(B,d)
\ar[r]^{\mathbf{G}}&\ar@<-3ex>[l]_{-\otimes^{\mathbb{L}}_{A}\mathbf{H}(A)}\ar@<3ex>[l]^{\mathbf{H}}\mathbf{D}(A,d)}.$$
Hence there is a triangle both in $\mathbf{D}(B,d)$ and $\mathbf{D}(B^{op},d)$:
$$\mathbf{G}(B)\otimes^{\mathbb{L}}_{A}\mathbf{H}(A)\rightarrow B\rightarrow i_{\ast}i^{\ast}(B_{B})\rightsquigarrow.$$
Applying the cohomology functor $H^{j}$ to this triangle, since $B$ is non-positive, one gets $H^{j}(i^{\ast}(B))\cong H^{j+1}(\mathbf{G}(B)\otimes^{\mathbb{L}}_{A}\mathbf{H}(A))\cong H^{j+1}(U\otimes^{\mathbb{L}}_{A}\mathbb{R}\mathrm{Hom}_{A}(U,A))$ for $j\geq 1$.

In the following we show that $\mathbb{R}\mathrm{Hom}_{A}(U,A)\simeq \mathbb{R}\mathrm{Hom}_{B^{op}}(U,B)$ in $\mathbf{D}(A^{op},d)$. In fact, since $U$ is a good $n$-silting object, there is a sequence of triangles in $\mathbf{D}(A,d)$
 $$A_{i}\longrightarrow U_{i}\longrightarrow A_{i+1}\rightsquigarrow~~ \text{with}~~ 0\leq i\leq n, $$
 such that $U_{i}\in \mathrm{add}(U)$, $A_{0}=A$ and $A_{n+1}=0$. Applying the functor $\Phi: \mathbb{R}\mathrm{Hom}_{A}(-,U_{A})$ to these triangles, we obtain another sequence of triangles
  $\Phi(A_{i+1})\longrightarrow \Phi(U_{i})\longrightarrow \Phi(A_{i})\rightsquigarrow~~ \text{with}~~ 0\leq i\leq n$. Therefore, we can construct the commutative diagram:
$$\scalebox{0.9}[0.9]{\xymatrix{
  \mathbb{R}\mathrm{Hom}_{A}(U,A_{n-1}) \ar[d]_{} \ar[r]^{} & \mathbb{R}\mathrm{Hom}_{A}(U,U_{n-1}) \ar[d]_{\simeq} \ar[r]^{} & \mathbb{R}\mathrm{Hom}_{A}(U,U_{n}) \ar[d]_{\simeq} \ar@{~>}[r]^{} &  \\
  \mathbb{R}\mathrm{Hom}_{B^{op}}(\Phi(A_{n-1}),\Phi(U)) \ar[r]^{} & \mathbb{R}\mathrm{Hom}_{B^{op}}(\Phi(U_{n-1}),\Phi(U)) \ar[r]^{} & \mathbb{R}\mathrm{Hom}_{B^{op}}(\Phi(U_{n}),\Phi(U)) \ar@{~>}[r]^{} &.    }}$$
 Note that we always have natural isomorphisms $B\cong\mathbb{R}\mathrm{Hom}_{A}(U,U)\cong \mathbb{R}\mathrm{Hom}_{B^{op}}(\Phi(U),\Phi(U))\cong\mathbb{R}\mathrm{Hom}_{B^{op}}(B,B)$. Then the isomorphisms in the second and third columns are due to  $U_{i}\in \mathrm{add}(U)$ for $0\leq i\leq n$. This implies that $ \mathbb{R}\mathrm{Hom}_{A}(U,A_{n-1})\cong \mathbb{R}\mathrm{Hom}_{B^{op}}(\Phi(A_{n-1}),\Phi(U))$. So we conclude that
$\mathbb{R}\mathrm{Hom}_{A}(U,A)\cong \mathbb{R}\mathrm{Hom}_{B^{op}}(\Phi(A),\Phi(U))=\mathbb{R}\mathrm{Hom}_{B^{op}}(U,B)$.
Thus, the equivalence follows from Lemma \ref{lem4.7}.
\end{proof}

If we specialising Theorem \ref{the4.8} to the case that $U$ is a good 1-silting object, then it is easy to check $H^{i}(U\otimes^{\mathbb{L}}_{A}\mathbb{R}\mathrm{Hom}_{A}(U,A))=0$ for $i\geq 2$, and we obtain the following corollary.
\begin{corollary}\label{cor4.6} Let $A$ be a dg-algebra, $U$ a good $1$-silting object
 in $\mathbf{D}(A,d)$, and let $B=\mathrm{DgEnd}_{A}(U)$. Then the dg-algebra $C$ induced in Theorem \ref{the4.5} is weak non-positive.
\end{corollary}

Let $R$ be a ring. From \cite[Lemma 5.5]{ch19}, if a left $R$-module $T$ is a good $n$-tilting module, then $T$ as a right $B$-module is an $n$-weak tilting module (see \cite[Definition 4.1]{ch19}), where $B$ is the endomorphism ring of $T$. Similarly, we introduce here the notation of $n$-weak silting objects, and show that if $U_{A}$ is a good $n$-silting object, then $_{B}U$ is $n$-weak silting whenever $A$ is weak non-positive, where $B=\mathrm{DgEnd}_{A}(U)$.

\begin{definition}\label{def4.9} Let $R$ be a dg-algebra.
An object $M$ in $\mathbf{D}(R^{op},d)$, or equivalently, a dg-$R^{op}$-module $M$ is called \emph{$n$-weak silting} provided that it satisfies the following conditions:

$(w1)$. dim$_{\mathrm{add}(R)} (M) \leq n$;

$(w2)$. M$^{(I)} \in M^{\perp_{>0}}$ for every set $I$,  and

$(w3)$. codim$_{\mathrm{Prod}(M)} R \leq n$.
\end{definition}

If an $n$-weak silting object $M$ satisfies $\mathrm{Prod}(_{R}M)=\mathrm{add}(_{R}M)$, then $_{R}M$ is a small $n$-silting object. On the other hand, small $n$-silting objects are always $n$-weak silting.

Let $S:=\mathrm{DgEnd}_{R^{op}}(M)$
\begin{align*}
 & \mathcal{Y}':=\{Y\in \mathbf{D}(R^{op},d)\mid
   \mathrm{Hom}_{\mathbf{D}(R^{op},d)}(M,Y[i])=0~\text{for all}~i\in \mathbb{Z} \}\\
 &G:={_{R}M\otimes^{\mathbb{L}}_{S}-}: \mathbf{D}(S^{op},d)\rightarrow \mathbf{D}(R^{op},d) \  \ \ \   \    H:=\mathbb{R}\mathrm{Hom}_{R^{op}}(M,-):\mathbf{D}(R^{op},d)\rightarrow \mathbf{D}(S^{op},d)
\end{align*}
Then $\mathcal{Y}'=\mathrm{Ker}(H)$. If $_{R}M$ satisfies $(w1)$, then $M$ is a compact and cofibrant object in $\mathbf{D}(R^{op},d))$. It follows that $\mathcal{Y}'$ is a bireflective subcategory of  $\mathbf{D}(R^{op},d))$ by \cite[Chapter III, Theorem 2.3; Chapter IV, Proposition 1.1]{be07}. If $_{R}M$ satisfies both (w1) and (w2), then $S$ is weak non-positive and the pair $(G,H)$ induces a triangle equivalence: $\mathbf{D}(S^{op},d)\overset{\simeq}{\rightarrow} \mathrm{Tria}(_{R}M)$. Thus we have the following result which shows that an $n$-weak silting object in $\mathbf{D}(R^{op},d)$ will always induce a recollement among derived categories of dg-algebras such that $\mathcal{Y}'$ is homological.

\begin{proposition}\label{pro4.10} Suppose the dg-$R^{op}$-module $M\in \mathbf{D}(R^{op},d)$  satisfies $(w1)$ and $(w2)$.  Then there exist a dg-algebra $C$ and a recollement of triangulated categories
$$\xymatrixcolsep{3pc}\xymatrix{\mathbf{D}(C^{op},d)\ar[r]^{\lambda_{\ast}}
&\ar@<-3ex>[l]_{}
\ar@<3ex>[l]^{}\mathbf{D}(R^{op},d)
\ar[r]^{H}&\ar@<-3ex>[l]_{G}\ar@<3ex>[l]^{\mathbb{R}\mathrm{Hom}_{S^{op}}(H(R),-)}\mathbf{D}(S^{op},d)}$$
such that $\lambda:R\rightarrow C$ is a homological epimorphism.
\end{proposition}
\begin{proof} Note that $M$ is a compact and cofibrant dg-$R^{op}$-module, by \cite[Lemma 3.2]{ba13}, there is a recollement:
$$\xymatrixcolsep{3pc}\xymatrix{\mathcal{Y}'\ar[r]^-{i_{\ast}=inc}
&\ar@<-3ex>[l]_-{i^{\ast}}
\ar@<3ex>[l]^-{i_{\ast}}\mathbf{D}(R^{op},d)
\ar[r]^{H}&\ar@<-3ex>[l]_{G}\ar@<3ex>[l]^{}\mathbf{D}(S^{op},d).}$$
 Let $\sigma:E\rightarrow R$ be a cofibrant replacement of $R$ in $\mathrm{HoDga}(k)$. Then $E$ is  $k$-flat and the restriction functor $\sigma_{\ast}: \mathbf{D}(R^{op},d)\rightarrow\mathbf{ D}(E^{op},d)$ is a triangle equivalence with $\sigma^{\ast}=-\otimes^{\mathbb{L}}_{E}R$ as a quasi-inverse. So by \cite[Theorem in \S4]{ni09} there exists a morphism of dg-algebras $f:E\rightarrow C$ such that $C\otimes^{\mathbb{L}}_{E}C\rightarrow C$ is a quasi-isomorphism. Hence $\lambda:=f\sigma^{-1}:R \rightarrow C$ is a homological epimorphism in $\mathrm{HoDga}(k)$. Then by the proof of Theorem \ref{the4.5}(1), we see that $\mathcal{Y}'$ is the essential image of the fully faithful functor $\sigma^{\ast}f_{\ast}:\mathbf{D}(C^{op},d)\rightarrow \mathbf{D}(R^{op},d)$. Thus we obtain the desired recollement.
\end{proof}

Let $A$ be a dg-algebra, $U\in \mathbf{D}(A,d)$ a good $n$-silting object, and let $B=\mathrm{DgEnd}_{A}(U)$.
Now we point out that each good silting object can produce a weak silting object over weak non-positive dg-algebras.

\begin{lemma}\label{lem4.11} If the dg-algebra $A$ is weak non-positive, then the object $U\in \mathbf{D}(B^{op},d)$ is $n$-weak silting.
\end{lemma}
\begin{proof}By the definition, $\mathrm{codim}_{\mathrm{add}(U)}A\leq n$. Then from \cite[Lemma 1.1]{br18}, we have dim$_{\mathrm{add}(B)}U\leq n$. By assumption, $A$ is weak non-positive, hence $\mathrm{Hom}_{\mathbf{D}(B^{op},d)}(U,U[i])\cong H^{i}A=0$ for $i\geq 1$. So ($w1$) and ($w2$) hold for $U$. Now, we check $(w3)$ for $U$. Since $\mathrm{dim}_{\mathrm{Add}(A)}U\leq n$, there is a sequence of triangles in $\mathbf{D}(A,d)$
 $$V_{i+1}\longrightarrow P_{i}\longrightarrow V_{i}\rightsquigarrow~~ \text{with}~~ 0\leq i\leq n, $$
 such that $P_{i}\in \mathrm{Add}(A)$, $V_{0}=U$ and $V_{n+1}=0$. In fact, applying the functor $\mathbb{R}\mathrm{Hom}_{A}(-,U)$ to these triangles, we get triangles in $\mathbf{D}(B^{op},d)$ of the form
 $$B_{i}\longrightarrow Q_{i}\longrightarrow B_{i+1}\rightsquigarrow~~ \text{with}~~ 0\leq i\leq n, $$
 such that $Q_{i}=\mathbb{R}\mathrm{Hom}_{A}(P_{i},U)\in \mathrm{Prod}(U)$, $B_{0}=B$ and $B_{n+1}=0$. Thus $U$ satisfies $(w3)$.
\end{proof}

From Proposition \ref{pro4.10} and Lemma \ref{lem4.11},  since $_{B}U$ is $n$-weak silting, there exsit a dg-algebra $E$ and a recollement among the derived categories $\mathbf{D}(A^{op},d)$, $\mathbf{D}(B^{op},d)$ and $\mathbf{D}(E^{op},d)$. In the end of this subsection, we consider when the dg-algebra $E$ is weak non-positive. By an argument similar to that in Lemma \ref{lem4.7}, we have the following lemma.

\begin{lemma} \label{lem4.12} Let $A$ be a dg-algebra, and let $p$ be the left adjoint of the inclusion $\iota:\mathcal{Y}'\rightarrow \mathbf{D}(B^{op},d)$. Then the induced  dg-algebra $E$ is weak non-positive if and only if $H^{i}(p(B))=0$ for every $i\geq1$.
\end{lemma}

\begin{theorem}\label{the4.13} Let $A$ be a weak non-positive dg-algebra, $U\in \mathbf{D}(A,d)$ a good $n$-silting object, and let $B=\mathrm{DgEnd}_{A}(U)$. Then there exist a dg-algebra $E$ and a recollement of triangulated categories
$$\xymatrixcolsep{3pc}\xymatrix{\mathbf{D}(E^{op},d)\ar[r]^{\lambda_{\ast}}
&\ar@<-3ex>[l]_{}
\ar@<3ex>[l]^{}\mathbf{D}(B^{op},d)
\ar[r]^{H}&\ar@<-3ex>[l]_{G}\ar@<3ex>[l]^{}\mathbf{D}(A^{op},d)}$$
such that $\lambda:B\rightarrow E$ is a homological epimorphism. Moreover, $E$ is weak non-positive if and only if $H^{i}(U\otimes^{\mathbb{L}}_{A}\mathbb{R}\mathrm{Hom}_{B^{op}}(U,B))=0$ for $i\geq 2$.
\end{theorem}
\begin{proof}From Lemma \ref{lem4.11}, $_{B}U$ is an $n$-weak silting dg-module, then it follows from Proposition \ref{pro4.10} that there exists a recollement of derived categories
$$\xymatrixcolsep{4pc}\xymatrix{\mathcal{Y}'\ar[r]^-{\iota=inc}
&\ar@<-3ex>[l]_-{p}
\ar@<3ex>[l]^-{q}\mathbf{D}(B^{op},d)
\ar[r]^{H}&\ar@<-3ex>[l]_{G}\ar@<3ex>[l]^{}\mathbf{D}(A^{op},d)}.$$
Hence there is a triangle both in $\mathbf{D}(B^{op},d)$ and $\mathbf{D}(B,d)$:
$$GH(B)\rightarrow {B}\rightarrow \iota p(B)\rightsquigarrow.$$
Applying the cohomology functor $H^{j}$ to this triangle, we get an exact sequence
$$\cdots\rightarrow H^{i}(p(B))\rightarrow H^{i+1}(GH(B))\rightarrow H^{i+1}(B)\rightarrow H^{i+1}(p(B))\rightarrow\cdots.$$
Since $B$ is non-positive, one gets $H^{j}(p(B))\cong H^{j+1}(U\otimes^{\mathbb{L}}_{A}\mathbb{R}\mathrm{Hom}_{B}(U,B))$ for $j>0$. Thus, the equivalence follows from Lemma \ref{lem4.12}.
\end{proof}
\begin{remark} It is easy to see that if the dg-algebra $A$ is weak non-positive, then the dg-algebra $C$ induced in Theorem \ref{the4.8} is weak non-positive if and only if the dg-algebra $E$ induced in Theorem \ref{the4.13} is weak non-positive.
\end{remark}

\bigskip
\section { \bf Applications }

In this section, we concern with some applications of the results in Section 4.

\bigskip
\textbf{5.1. Applications to good cosilting objects over weak non-positive dg-algebras }
\bigskip

In this subsection, we shall apply the results in Section 4 to deal with good cosilting objects.
First, we construct $n$-weak silting objects from good $n$-cosilting objects, and then use
Proposition \ref{pro4.10} to construct the recollement.

Let $(A, d)$ be a weak non-positive dg-algebra. By \cite[3.4.3]{av03}, there exists a functor
 ${\vee}: \mathrm{Mod}(A,d)\rightarrow \mathrm{Mod}(A^{op},d)$. Set $W:=A_{A}^{\vee}$. Then by \cite[Proposition 10.1.4]{av03}, $W$ is a fibrant dg-$A^{op}$ module.

 \begin{definition}
An object $N$ in $\mathbf{D}(A^{op},d)$, or equivalently, a dg-$A^{op}$-module $N$ is called \emph{ $n$-cosilting} if it satisfies the following conditions:

$(C1)$. codim$_{\mathrm{Prod}(W)}(_{A}N )\leq n$;

$(C2)$. N$^{I} \in {^{\perp_{>0}}N}$ for every set $I$,  and

$(C3)$. dim$_{\mathrm{Prod}(N)} W \leq n$.

An  $n$-cosilting dg-$A^{op}$-module $N$ is said to be \emph{good} if it satisfies $(C1)$, $(C2)$ and

$(c3)$. dim$_{\mathrm{add}(N)} W \leq n$.
\end{definition}

We say that $N$ is a (good) \emph{cosilting} dg-$A^{op}$-module if $_{A}N$ is (good) $n$-{cosilting} for some $n\in \mathbb{N}$.

From now on, we assume that $N$ is a good $n$-cosilting dg-$A^{op}$-module with $(C1)$, $(C2)$ and $(c3)$, and call $N$ a good $n$-cosilting dg $A^{op}$-module with respect to $W$. Let $B := \mathbb{R}\mathrm{Hom}_{A}(N,N)$,
$M := \mathbb{R}\mathrm{Hom}_{A}(N,W)$ and $\Lambda := \mathrm{DgEnd}_{A}(W)$. Then $M$ is a dg-$B^{op}$-$\Lambda$-module.

\begin{lemma}\label{lem5.2}
(1) $\mathrm{dim}_{\mathrm{add}(B)}M\leq n$.

(2) The functor $\mathbb{R}\mathrm{Hom}_{A}(N,-) : \mathbf{D}(A^{op},d)\rightarrow \mathbf{D}(B^{op},d)$ induces an isomorphism of
dg-algebras: $\Lambda \simeq \mathbb{R}\mathrm{Hom}_{B}(M,M)$, and $\mathrm{Hom}_{\mathbf{D}(B^{op},d)}(M,M[i])=0$ for all $i\geq 1$.

(3) The dg-$B^{op}$-module $M$ satisfies $(w1)$-$(w3)$ in Definition \ref{def4.9}.
\end{lemma}
\begin{proof} (1) Since dim$_{\mathrm{add}(N)} W \leq n$, there is a sequence of triangles in $\mathbf{D}(A^{op},d)$
 $$K_{i+1}\longrightarrow N_{i}\longrightarrow K_{i}\rightsquigarrow~~ \text{with}~~ 0\leq i\leq n, \eqno(5.1)$$
 such that $N_{i}\in \mathrm{add}(N)$, $K_{0}=W$ and $V_{n+1}=0$. Applying $\mathbb{R}\mathrm{Hom}_{A}(N,-)$ to these triangles, we get a sequence of triangles in $\mathbf{D}(B^{op},d)$ of the form
 $$V_{i+1}\longrightarrow B_{i}\longrightarrow V_{i}\rightsquigarrow~~ \text{with}~~ 0\leq i\leq n, $$
 such that $B_{i}=\mathbb{R}\mathrm{Hom}_{A}(N,N_{i})\in \mathrm{add}(B)$, $V_{0}=M$ and $V_{n+1}=0$. Thus dim$_{\mathrm{add}(B)}M\leq n$.

(2) Let $\Psi$ be the functor $\mathbb{R}\mathrm{Hom}_{A}(N,-) : \mathbf{D}(A^{op},d)\rightarrow \mathbf{D}(B^{op},d)$. Then $\Psi(N) = B$,
$\Psi(W) = M$ and $\mathbb{R}\mathrm{Hom}_{A}(X,W)\overset{\sim}{\rightarrow}\mathbb{R}\mathrm{Hom}_{B}(\Psi(X), \Psi(W))$ for any $X\in \mathrm{add}(_{A}N)$.

If $n = 0$, then $W = N_{0}$ and $M = \mathbb{R}\mathrm{Hom}_{A}(N,N_{0})$. In this case, one can
easily check (2).

Suppose $n \geq 1$. By (1), we have  a sequence of triangles in $\mathbf{D}(B^{op},d)$ of the form
 $$V_{i+1}\longrightarrow B_{i}\longrightarrow V_{i}\rightsquigarrow~~ \text{with}~~ 0\leq i\leq n, $$
 such that $B_{i}=\mathbb{R}\mathrm{Hom}_{A}(N,N_{i})\in \mathrm{add}(B)$, $V_{0}=M$ and $V_{n+1}=0$.  Applying $\mathbb{R}\mathrm{Hom}_{A}(-,W)$ to the triangles (5.1),
  we get  a sequence of triangles:
 $$\mathbb{R}\mathrm{Hom}_{A}(K_{i},W)\rightarrow \mathbb{R}\mathrm{Hom}_{A}(N_{i},W)\rightarrow \mathbb{R}\mathrm{Hom}_{B}(K_{i+1},W)\rightsquigarrow ~~ \text{with}~~ 0\leq i\leq n.$$
We can construct the commutative diagram:
$$\scalebox{0.9}[0.9]{\xymatrix{
  \mathbb{R}\mathrm{Hom}_{A}(K_{n-1},W) \ar[d]_{} \ar[r]^{} & \mathbb{R}\mathrm{Hom}_{A}(N_{n-1},W) \ar[d]_{\simeq} \ar[r]^{} & \mathbb{R}\mathrm{Hom}_{A}(N_{n},W) \ar[d]_{\simeq} \ar@{~>}[r]^{} &  \\
  \mathbb{R}\mathrm{Hom}_{B}(\Psi(K_{n-1}),\Psi(W)) \ar[r]^{} & \mathbb{R}\mathrm{Hom}_{B}(\Psi(N_{n-1}),\Psi(W)) \ar[r]^{} & \mathbb{R}\mathrm{Hom}_{B}(\Psi(N_{n}),\Psi(W)) \ar@{~>}[r]^{} &.    }}$$
This implies that $ \mathbb{R}\mathrm{Hom}_{A}(K_{n-1},W)\cong \mathbb{R}\mathrm{Hom}_{B}(\Psi(K_{n-1}),\Psi(W))$. So we conclude that
$\Lambda\cong\mathbb{R}\mathrm{Hom}_{A}(W,W)\cong \mathbb{R}\mathrm{Hom}_{B}(\Psi(W),\Psi(W))=\mathbb{R}\mathrm{Hom}_{B}(M,M)$.

It remains to prove that $\mathrm{Hom}_{\mathbf{D}(B^{op},d)}(M,M[i])=0$ for all $i\geq 1$. We claim that if $A$ is weak non-positive, then so is $\Lambda$. In fact, since $W=A_{A}^{\vee}$ is a fibrant dg-$A^{op}$-module, from \cite[Proposition 3.4.7]{av03}, there are isomorphisms $\Lambda=\mathrm{Hom}^{\bullet}_{A}(A^{\vee},A^{\vee})\cong \mathrm{Hom}^{\bullet}_{A}(A,A^{\vee\vee})\cong A^{\vee\vee}$. Therefore, the claim follows from the isomorphism $H^{i}(A^{\vee\vee})\cong H^{i}(A)^{\vee\vee}$  by \cite[Lemma 3.4.4]{av03}, where $i\in \mathbb{Z}$. Thus, $\mathrm{Hom}_{\mathbf{D}(B^{op},d)}(M,M[i])\cong H^{i}\mathbb{R}\mathrm{Hom}_{B}(M,M)\cong H^{i}\Lambda =0$ for all $i\geq 1$.

(3) Clearly, $(w1)$ and $(w2)$ follow from (1) and (2), respectively. It remains to show
$(w3)$ for $M$. In fact, by ($C1$), there exists  a sequence of triangles $$U_{i}\longrightarrow I_{i}\longrightarrow U_{i+1}\rightsquigarrow~~ \text{with}~~ 0\leq i\leq n, $$
 such that $I_{i}\in \mathrm{Prod}(W)$, $U_{0}=N$ and $U_{n+1}=0$.
Applying $\mathbb{R}\mathrm{Hom}_{A}(N,-)$ to these triangles,
  we get  a sequence of triangles:
 $$\mathbb{R}\mathrm{Hom}_{A}(N,U_{i})\rightarrow \mathbb{R}\mathrm{Hom}_{A}(N,I_{i})\rightarrow \mathbb{R}\mathrm{Hom}_{B}(N,U_{i+1})\rightsquigarrow ~~ \text{with}~~ 0\leq i\leq n.$$
Since $\mathbb{R}\mathrm{Hom}_{A}(N,-)$  commutes with arbitrary direct products, it follows from $I_{i}\in \mathrm{Prod}(W)$ that $\mathbb{R}\mathrm{Hom}_{A}(N,I_{i})\in ¡Ê \mathrm{Prod}(\mathbb{R}\mathrm{Hom}_{A}(N,W)) = \mathrm{Prod}(M)$ and that $_{A}M$ satisfies ($w3$).
\end{proof}

By Lemma \ref{lem5.2}(2), the dg-algebra $\mathbb{R}\mathrm{Hom}_{B}(M,M)$ can be identified naturally with $\Lambda$. Now, we define
$G := {_{B}M_{\Lambda}\otimes^{\mathbb{L}}_{\Lambda}-}:\mathbf{D}(\Lambda^{op},d)\rightarrow \mathbf{D}(B^{op},d)$ and $H := \mathbb{R}\mathrm{Hom}_{B}(M,-) : \mathbf{D}(B^{op},d)\rightarrow\mathbf{D}(\Lambda^{op},d)$.
Since $_{B}M$ satisfies both ($w1$) and ($w2$) in Definition \ref{def4.9}, by Proposition \ref{pro4.10} and the proof of Theorem \ref{the4.13},  there exist a dg-algebra $C$ and a recollement of triangulated categories
$$\xymatrixcolsep{3pc}\xymatrix{\mathbf{D}(C^{op},d)\ar[r]^{\lambda_{\ast}}
&\ar@<-3ex>[l]_{}
\ar@<3ex>[l]^{}\mathbf{D}(B^{op},d)
\ar[r]^{H}&\ar@<-3ex>[l]_{G}\ar@<3ex>[l]^{}\mathbf{D}(\Lambda^{op},d).}$$
such that $\lambda:B\rightarrow C$ is a homological epimorphism. Moreover, $C$ is weak non-positive if and only if $H^{i}(M\otimes^{\mathbb{L}}_{\Lambda}\mathbb{R}\mathrm{Hom}_{B^{op}}(M,B))=0$ for $i\geq 2$.

\bigskip
\textbf{5.2. Applications to good 2-term silting complexes}
\bigskip

In this subsection, we show that there exists a recollement induced by good 2-term silting complexes.

Let $R$ be an ordinary ring,  $\mathbb{P}$ the complex
$$\cdots\rightarrow0\rightarrow P^{-1} \stackrel{\sigma}\rightarrow P^{0}\rightarrow0\rightarrow\cdots$$
 with $P^{-1}, P^{0}$ projective. From \cite{bre18}, $\mathbb{P}$ is called a \emph{good 2-term silting complex} if

(S1) $\mathbb{P}^{(I)}\in \mathbb{P}^{\perp_{>0}}$ for all sets $I$,
where
$$\mathbb{P}^{\perp_{>0}} = \{Y \in \mathbf{D}(R)\mid \mathrm{Hom}_{\mathbf{D}(R)}(\mathbb{P}, Y [n]) = 0~\text{for ~all~ positive~ integers}~ n\}.$$

(S2) there exists a triangle
$$R\rightarrow \mathbb{P}^{n}\rightarrow  \mathbb{P}' \rightsquigarrow$$
in $\mathbf{D}(R)$ such that $\mathbb{P}' \in \mathrm{add}(\mathbb{P})$.

(S3) the homotopy category $\mathbf{K}^{b}(\mathrm{Proj}~R)$ of
bounded complexes of projective modules is the smallest triangulated subcategory of $\mathbf{D}(R)$ containing Add$(\mathbb{P})$.

Denote by $B$ the smart truncation of $\mathrm{DgEnd}_{R}(\mathbb{P})$ as in the Section 4.1. Then $B$ is a non-positive dg-algebra and we have a quasi-isomorphism $B \rightarrow \mathrm{DgEnd}_{R}(\mathbb{P})$. Note that $\mathbb{P}\in\mathbf{K}^{b}(\mathrm{Proj}~R)=\langle \mathrm{Add}(R)\rangle$ and $\mathbb{P}\in R^{\perp_{>0}}$. Hence $\mathrm{dim}_{\mathrm{Add}(R)}\mathbb{P}<\infty $ by \cite[Corollary 2.6 (2)]{we13}. Furthermore, from  \cite[Proposition 3.9]{we13}, we see that the good 2-term silting complex is a 1-silting object in $\mathbf{D}(R,d)$.  Thus, as a consequence of Theorems \ref{the4.3} and \ref{the4.5}, we obtain the next recollement.
\begin{corollary}
Let $R$ be a $k$-algebra, $\mathbb{P}$ a good 2-term silting complex in $\mathbf{D}(R)$.  Then there is a homological epimorphism $g=f\sigma^{-1}:B \rightarrow C$ in $\mathrm{HoDga}(k)$, represented by homomorphisms of dg-algebras $\sigma:E\rightarrow B$ and $f:E\rightarrow C$, such that the following is a recollement of triangulated categories:
$$\xymatrixcolsep{4pc}\xymatrix{\mathbf{D}(C,d)\ar[r]^{\sigma^{\ast}f_{\ast}}&\ar@<-3ex>[l]_{C\otimes_{E}^{\mathbb{L}}-}
\ar@<3ex>[l]^{\mathbb{R}\mathrm{Hom}_{E}(C, -)}\mathbf{D}(B,d)
\ar[r]^{}&\ar@<-3ex>[l]_{}\ar@<3ex>[l]^{}\mathbf{D}(R).}$$
Moreover, $C$ is weak non-positive.
\end{corollary}

\bigskip
\textbf{5.3. Applications to good tilting complexes and modules}
\bigskip

In the end of this section, we want to show that our results generalize those of \cite{ch12,ch19}. In order to do that, let $R$ be a ring and $T$ an $R$-complex.
Then $R$ can be seen as a dg-algebra concentrated in degree 0.
Recall that $\mathbf{K}^{b}(\mathrm{Proj}~R)$ denotes the homotopy category of
bounded complexes of projective modules.
The complex $T$ is called a \emph{good
tilting complex} if  it satisfies the following conditions:

($T1$). $T\in \mathbf{K}^{b}(\mathrm{Proj}~R)$;

($T2$).  $\mathrm{Hom}_{\mathbf{D}(R)}(T, T^{(\alpha)}[n]) =
0$ for every set $\alpha$ and $ n\neq 0$.

($t3$). codim$_{\mathrm{add}(T)} R < \infty$.

One can check that $\mathbf{K}^{b}(\mathrm{Proj}~R)=\langle \mathrm{Add}(R)\rangle$. Hence the good tilting complexes are good silting objects in $\mathbf{D}(R)$. From the condition ($T2$), $H^{n}(\mathrm{DgEnd}_{R}(U))=H^{n}(\mathbb{R}\mathrm{Hom}_{R}(U,U))\cong \mathrm{Hom}_{\mathbf{D}(R)}(T, T[n]) =0$ for  $ n\neq 0$. Since $H^{0}(\mathrm{DgEnd}_{R}(U))\cong B$, we get $\mathbf{D}(\mathrm{DgEnd}_{R}(U),d)=\mathbf{D}(B)$. As a consequence of Theorems \ref{the4.5} and \ref{the4.8}, we obtain the folllowing recollement.

\begin{corollary}Let $R$ be an ring, and $T$ a good tilting complex, and let $B=\mathrm{End}(T)$.  Then there exist a dg-algebra $C$ and a
recollement of triangulated categories
$$\xymatrixcolsep{3pc}\xymatrix{\mathbf{D}(C,d)\ar[r]^{\lambda_{\ast}}
&\ar@<-3ex>[l]_{}
\ar@<3ex>[l]^{}\mathbf{D}(B)
\ar[r]^{\mathbf{G}}&\ar@<-3ex>[l]_{-\otimes^{\mathbb{L}}_{R}\mathbf{H}(R)}\ar@<3ex>[l]^{}\mathbf{D}(R)}$$
such that $\lambda:B\rightarrow C$ is  a homological epimorphism.
Moreover, $C$ is weak non-positive if and only if $H^{i}(T\otimes^{\bullet}_{R}\mathrm{Hom}^{\bullet}_{R}(T,R))=0$ for all
$i\geq 2$.
\end{corollary}

Let $R$ be a ring and $T$ an $R$-module. Consider the following
conditions on $T$:

(T1) The projective dimension of $T$ is finite;

(T2) The module $T$ has no self-extensions, that is $\mathrm{Ext}^{i}_{R}(T, T^{(\alpha)}) =
0$ for every $ i\geq 1$ and every set $\alpha$.

(T3) There is an exact sequence of $R$-modules
$$0 \rightarrow R\rightarrow T_{0}\rightarrow T_{1} \rightarrow \cdots \rightarrow T_{n} \rightarrow 0$$
such that $T_{i}$ is isomorphic to a direct summand of a finite direct sum of copies of $T$ for
all $0\leq i \leq n$.

Then $T$ is called a \emph{good $n$-tilting module}, if it satisfies (T1), (T2), (T3) and the projective dimension of $T$ is at most $n$. Let $B$ be the endomorphism ring of $T$. Chen and Xi \cite[Theorem 1.1]{ch12} proved that if $T$ is a good $1$-tilting module, then there exists a ring $C$, a homological ring epimorphism $B\rightarrow C$ and a recollement among the (unbounded) derived module categories $\mathbf{D}(C)$ of $C$, $\mathbf{D}(B)$ of $B$ and $\mathbf{D}(A)$ of $A$. Recently, in \cite{ch19}, they give a necessary
and sufficient condition for good tilting modules of higher projective dimension to induce
recollements of derived module categories via homological ring epimorphisms.

We obtain the next recollement. One can compare it with \cite[Theorem 1.1]{ch19}.

\begin{theorem}\label{the5.5}Let $R$ be an ordinary algebra, $T_{R}$ a good tilting module, and let $B$ be the endomorphism ring of $T$.  Then there exist a  dg-algebra $C$ and a
recollement of triangulated categories
$$\xymatrixcolsep{3pc}\xymatrix{\mathbf{D}(C,d)\ar[r]^{\lambda_{\ast}}
&\ar@<-3ex>[l]_{}
\ar@<3ex>[l]^{}\mathbf{D}(B)
\ar[r]^{\mathbf{G}}&\ar@<-3ex>[l]_{-\otimes^{\mathbb{L}}_{R}\mathbf{H}(R)}\ar@<3ex>[l]^{}\mathbf{D}(R)}$$
such that $\lambda:B\rightarrow C$ is  a homological epimorphism.
Moreover, we can choose $C$ to be an ordinary $k$-algebra if and only if $H^{i}(T\otimes_{R}\mathrm{Hom}^{\bullet}_{R}(U,R))=0$ for all
$i\geq 2$, where the complex $U$ is a deleted projective resolution of $T$.
\end{theorem}
\begin{proof} Denote by $U$ a deleted projective resolution of $T$. Then $T=H^{0}(U)$ and $U\in \mathbf{D}(R)$ is a good silting object. It is easy to see that $\mathbf{D}(\mathrm{DgEnd}_{R}(U),d)=\mathbf{D}(B)$ and then the recollement follows by Theorem \ref{the4.8}.

Note that $U_{R}\rightarrow T_{R}$ is a quasi-isomorphism and $\mathrm{Hom}_{R}^{\bullet}(U,-)$ preserves quasi-isomorphism. This implies that
$\mathrm{Hom}_{R}^{\bullet}(U,U)\cong\mathrm{Hom}_{R}^{\bullet}(U,T)$ in $\mathbf{D}(B)$. Denote  $X:=\mathrm{Hom}_{R}^{\bullet}(U,T)$. Since each term of the complex $U$ belongs to $\mathrm{Add}(R)$, it follows that each term of the complex $X$ belongs to $\mathrm{Prod}(T_{R})$.
Furthermore, since $T_{R}$ is a good tilting module, it is shown in \cite[Lemma 5,5]{ch19}  that $T$ is a strongly $R$-Mittag-Leffler module(see \cite[Definition 4.1]{ch19} or \cite[Definition 1.1]{li08}). Therefore, by an argument similar to that in \cite[Lemma 4.3]{ch19}, we see that $$U\otimes^{\mathbb{L}}_{R}\mathbb{R}\mathrm{Hom}_{R}(U,R)\cong T\otimes^{\mathbb{L}}_{R}\mathbb{R}\mathrm{Hom}_{B^{op}}(T,B)\cong T\otimes_{R}\mathrm{Hom}^{\bullet}_{B^{op}}(T,\mathrm{Hom}_{R}^{\bullet}(U,T)).$$
From the proof of \cite[Theorem 1.1]{ch19}, there exists an isomorphism $\mathrm{Hom}^{\bullet}_{B^{op}}(T,\mathrm{Hom}_{R}^{\bullet}(U,T))\cong\mathrm{Hom}_{R}^{\bullet}(U,R)$. Consequently, by Theorem \ref{the4.8}, $C$ is weak non-positive if and only if $H^{i}(T\otimes_{R}\mathrm{Hom}^{\bullet}_{R}(U,R))=0$ for all $i\geq 2$.

Next, we claim that in this case, the dg-algebra $C$ is always weak positive. Thus $C$ can be replaced  by an ordinary $k$-algebra.
In fact, since $T$ is a strongly $R$-Mittag-Leffler  and each term of the complex $X$ belongs to $\mathrm{Prod}(T_{R})$, we obtain a short exact sequence of complexes:
$$0\rightarrow T\otimes_{R}\mathrm{Hom}^{\bullet}_{B^{op}}(T,X)\rightarrow X\rightarrow L\rightarrow 0,$$
where the the terms of complex $L$ are concentrated in degrees $\geq0$.
Hence we can construct the commutative diagram both in $\mathbf{D}(B)$ and $\mathbf{D}(B^{op})$:
$$\scalebox{1}[1]{\xymatrix{
 U\otimes^{\mathbb{L}}_{R}\mathbb{R}\mathrm{Hom}_{R}(U,R) \ar[d]_{\simeq} \ar[r]^{} & B \ar@{=}[d]_{} \ar[r]^{} & i_{\ast}i^{\ast}B \ar@{=}[d]_{} \ar@{~>}[r]^{} &  \\
 T\otimes^{\mathbb{L}}_{R}\mathbb{R}\mathrm{Hom}_{B^{op}}(T,B) \ar[d]_{\simeq} \ar[r]^{} & B \ar[d]_{\simeq} \ar[r]^{} & i_{\ast}i^{\ast}B \ar[d]_{\simeq} \ar@{~>}[r]^{} &  \\
  T\otimes_{R}\mathrm{Hom}^{\bullet}_{B^{op}}(T,X) \ar[r]^{} & X \ar[r]^{} & L\ar@{~>}[r]^{} &.    }}$$
In particular, $H^{j}(i^{\ast}B)\cong H^{j}(L)=0$ for all $j<0$. Hence by the proof of Lemma \ref{lem4.7}, we deduce that $C$ is always weak positive.
\end{proof}
If we specialising Theorem \ref{the5.5} to the case that $T$ is a good 1-tilting module, then it is easy to check that $H^{i}(T\otimes_{R}\mathrm{Hom}^{\bullet}_{R}(U,R))=0$ for all
$i\geq 2$, and we get the following corollary, which recover the results in \cite[Theorem 1.1]{ch12}.

\begin{corollary}Let $A$ be an ring, $T$ a good 1-tilting $A$-module and $B$ the endomorphism ring of $T$. Then there is a ring $C$, a homological ring epimorphism $\lambda : B \rightarrow C$ and a
recollement among the unbounded derived categories of the rings $A$, $B$ and $C$:
$$\xymatrixcolsep{3pc}\xymatrix{\mathbf{D}(C)\ar[r]^{\lambda_{\ast}}
&\ar@<-3ex>[l]_{}
\ar@<3ex>[l]^{}\mathbf{D}(B)
\ar[r]^{j^{!}}&\ar@<-3ex>[l]_{}\ar@<3ex>[l]^{}\mathbf{D}(A)}$$
such that the triangle functor $j^{!}$  is isomorphic to the total left-derived functor $-\otimes^{\mathbb{L}}_{B}T$.
\end{corollary}

\bigskip \centerline {\bf Acknowledgements}
\bigskip
\hspace{-0.5cm}  This research was partially supported by the National Natural Science Foundation of China (Nos. 11771202, 11771212). The main part of this work was carried out while the first author was visiting  Charles University in Prague. He gratefully acknowledges the financial support from China Scholarship Council (CSC No. 201806190107) and the kind hospitality from the host university.

\end{document}